\documentclass[12pt, a4paper]{amsart}
\usepackage{mathtools} 
\mathtoolsset{showonlyrefs,showmanualtags}
\usepackage[pagebackref]{hyperref} 
\hypersetup{
colorlinks=true,       
linkcolor=blue,          
citecolor=magenta,        
filecolor=magenta,      
urlcolor=cyan           
}

\usepackage[normalem]{ulem}

\allowdisplaybreaks

\usepackage{amsmath,amsthm,amscd,amssymb,amsfonts, amsbsy,mathrsfs}
\usepackage{latexsym}
\usepackage{exscale}

\usepackage[utf8]{inputenc}

\calclayout

\numberwithin{equation}{section}

\theoremstyle{plain}
\newtheorem{Theorem}[equation]{Theorem}

\newtheorem{Proposition}[equation]{Proposition}
\newtheorem{Corollary}[equation]{Corollary}
\newtheorem{Lemma}[equation]{Lemma}

\theoremstyle{definition}
\newtheorem{Remark}[equation]{Remark}

\theoremstyle{definition}

\def\Xint#1{\mathchoice
{\XXint\displaystyle\textstyle{#1}}%
{\XXint\textstyle\scriptstyle{#1}}%
{\XXint\scriptstyle\scriptscriptstyle{#1}}%
{\XXint\scriptscriptstyle\scriptscriptstyle{#1}}%
\!\int}
\def\XXint#1#2#3{{\setbox0=\hbox{$#1{#2#3}{\int}$}
\vcenter{\hbox{$#2#3$}}\kern-.5\wd0}}

\def\dashint{\Xint-}
\def\d{d}

\newcommand{\BMO}{\mathcal{BMO}}

\newcommand{\re}{\mathbb{R}}

\DeclareMathOperator*{\essinf}{ess\,inf}
\DeclareMathOperator*{\esssup}{ess\,sup}

\begin{document}

\title[Extrapolation for multilinear Muckenhoupt classes]{Extrapolation for multilinear Muckenhoupt classes and applications to the bilinear Hilbert transform}

\author{Kangwei Li} 
\address{Kangwei Li, BCAM, Basque Center for Applied Mathematics, Mazarredo 14, 48009
	Bilbao, Basque Country, Spain}
\email{kli@bcamath.org}

\author{Jos\'e Mar{\'\i}a Martell}
\address{Jos\'e Mar{\'\i}a Martell
	\\
	Instituto de Ciencias Matemáticas CSIC-UAM-UC3M-UCM
	\\
	Consejo Superior de Investigaciones Científicas
	\\
	C/ Nicol\'as Cabrera, 13-15
	\\
	E-28049 Madrid, Spain} 

\email{chema.martell@icmat.es}

\author{Sheldy Ombrosi}
\address{Sheldy Ombrosi, Department of Mathematics, Universidad Nacional del Sur,
Bah\'ia Blanca, Argentina}
\email{sombrosi@uns.edu.ar}
\address{and}
\address{
BCAM, Basque Center for Applied Mathematics, Mazarredo 14, 48009
	Bilbao, Basque Country, Spain}
\email{sombrosi@bcamath.org}

\thanks{The first author is supported by Juan de la Cierva - Formaci\'on 2015 FJCI-2015-24547, by the Basque Government through the BERC
2018-2021 program and by Spanish Ministry of Economy and Competitiveness
MINECO through BCAM Severo Ochoa excellence accreditation SEV-2013-0323
and through project MTM2017-82160-C2-1-P funded by (AEI/FEDER, UE) and
acronym ``HAQMEC''.
	The second author acknowledges financial
	support from the Spanish Ministry of Economy and Competitiveness,
	through the ``Severo Ochoa'' Programme for Centres of Excellence in
	R\&D” (SEV-2015-0554). He also acknowledges that the research
	leading to these results has received funding from the European
	Research Council under the European Union's Seventh Framework
	Programme (FP7/2007-2013)/ ERC agreement no. 615112 HAPDEGMT. 
	The third author is supported by CONICET PIP 11220130100329CO, Argentina, by the Basque Government through the BERC
2018-2021 program and by Spanish Ministry of Economy and Competitiveness
MINECO through BCAM Severo Ochoa excellence accreditation SEV-2013-0323. 
The authors express their gratitude to C.~Benea, C.~Muscalu, and  R.~Torres for their comments. 
	Finally the first and second authors would  like to thank the last author for his hospitality during their visit to Bahía Blanca where much of
	the work on the project was done.}



\date{\today}


\begin{abstract}
In this paper we solve a long standing problem about the multivariable Rubio de Francia extrapolation theorem for the multilinear Muckenhoupt classes $A_{\vec{p}}$, which were extensively studied by Lerner \textit{et al.} and which are the natural ones for the class of multilinear Calder\'on-Zygmund operators. Furthermore, we go beyond the classes $A_{\vec{p}}$ and extrapolate within the classes $A_{\vec{p},\vec{r}}$ which appear naturally associated to the weighted norm inequalities for multilinear sparse forms which control fundamental operators such as the bilinear Hilbert transform. We give several applications which can be easily obtained using extrapolation. First, for the bilinear Hilbert transform one can extrapolate from the recent result of Culiuc \textit{et al.} who considered the Banach range and extend the estimates to the quasi-Banach range. As a direct consequence, we obtain weighted vector-valued inequalities reproving some of the results by Benea and Muscalu.  We also extend recent results of Carando \textit{et al.}~on Marcinkiewicz-Zygmund estimates for multilinear Calder\'on-Zygmund operators. Finally, our last application gives new weighted estimates for the commutators of multilinear Calder\'on-Zygmund operators  and for the bilinear Hilbert transform with BMO functions using ideas from Bényi \textit{et al.}
	
\end{abstract}

\maketitle

\section{Introduction}
The Rubio de Francia extrapolation theorem \cite{RdF} is a powerful tool in harmonic analysis which states that if a given operator $T$ is bounded on $L^{p_0}(w)$ for some fixed $p_0$, $1\le p_0<\infty$, and for all $w\in A_{p_0}$, then $T$ is indeed bounded on all $L^p(w)$ for all $1<p<\infty$ and for all $w\in A_p$. This is quite practical as it suffices to choose some particular exponent which could be natural for the operator in question and establish the weighted estimates for it. This, for instance, allows one to immediately obtain vector valued weighted estimates showing that $T$ is bounded on $L^p_{\ell^s}(w)$ for all $1<p,s<\infty$ and all $w\in A_p$. On the other hand, the classical Rubio de Francia extrapolation theorem is only suitable for operators for which we know or expect to have estimates in the range $(1,\infty)$. There are other versions considering estimates within the smaller classes of weights of the form $A_{\frac{p}{p_-}}\cap RH_{\big(\frac{p_+}{p}\big)'}$ which are naturally adapted to range $p\in (p_-, p_+)$ (see \cite{AM}) or off-diagonal extrapolation results for the classes $A_{p,q}$ which are natural for estimates from $L^p$ to $L^q$ where $p\neq q$ (see \cite{HMS}). We refer the reader to \cite{CMP}  for the development of extrapolation and a more complete list of references (see also \cite{Duo}). 

In the multivariable setting there are some Rubio de Francia extrapolation results. In \cite{grafakos-martell04} it was shown that if $T$ is bounded from $L^{p_1}(w_1)\times \dots \times L^{p_m}(w_m)$ to $L^p(w_1^{\frac{p}{p_1}} \dots w_m^{\frac{p}{p_m}})$ for some fixed exponents $1<p_1,\dots, p_m<\infty$, $\frac1p=\frac1{p_1}+\dots+\frac1{p_m}$, and for all  $w_i\in A_{p_i}$, then the same holds for all possible values of $p_j$. Much as before, this extrapolation result for products of Muckenhoupt classes is adapted to the ranges $p_j\in (1,\infty)$ and the recent paper \cite{CM-extrapol} extended extrapolation to the classes of  weights $w_j\in A_{\frac{p_j}{p_j^-}}\cap RH_{\big(\frac{p_j^+}{p}\big)'}$ which are associated with the ranges $p_j\in(p_j^-,p_j^+)$. These results are very natural extensions of the Rubio de Francia extrapolation theorem, but they treat each variable separately with its own Muckenhoupt class of weights (this fact also appears in the proofs, see \cite{Duo} or \cite{CM-extrapol}) and do not quite use the multivariable nature of the problem. In this direction \cite{LOPTT} introduced some multilinear Muckenhoupt classes. Namely, given 
$\vec{p}=(p_1,\dots,p_m)$ with $1\le p_1,\dots,p_m<\infty$, one says that $\vec{w}=(w_1,\dots, w_m)\in A_{\vec{p}}$ provided $0<w_i<\infty$ a.e. for every $i=1,\dots,m$ and
\[
[\vec{w}]_{A_{\vec p}}
=\sup_Q\Big(\dashint_Q w\d x\Big)^{\frac 1{p}}
\prod_{i=1}^m \Big(\dashint_Q w_i^{1-p_i'}\d x\Big)^{\frac1{p_i'}}<\infty,
\]
where $\frac1p=\frac1{p_1}+\dots+\frac1{p_m}$ and $w=w_1^{\frac{p}{p_1}} \dots w_m^{\frac{p}{p_m}}$ (here, when $p_i=1$, the term
corresponding to $w_i$ needs to be replaced by $\esssup_Q w_i^{-1}$). These classes of weights contain some multivariable structure in their definition and, as a matter of fact, characterize the boundedness from $L^{p_1}(w_1)\times \dots \times L^{p_m}(w_m)$ into $L^p(w)$ (where one has to replace $L^p(w)$ with $L^{p,\infty}(w)$ when at least one $p_i=1$) of the multi-sublinear  Hardy-Littlewood maximal function
\[
\mathcal{M}(f_1,\dots,f_m)(x)=\sup_{Q\ni x} \prod_{i=1}^m\Big(\dashint_Q |f_i(y)|dy\Big).
\]
Notice that $\mathcal{M}(f_1,\dots,f_m)(x)\lesssim M f_1(x)\dots M f_m(x)$, where $M$ is the regular Har\-dy-Little\-wood maximal function, and hence $A_{p_1}\times\dots\times A_{p_m}\subset A_{\vec{p}}$, however the inclusion is strict. This indicates that the multivariable operators are generally speaking richer than the simple multiplication of operators in each component. 

One of the main goals of this paper is to establish a Rubio de Francia extrapolation theorem valid in the context of the multilinear classes $A_{\vec{p}}$. Before stating the precise result we find illustrative to present some easy examples of weights which shed light on the structure of the multilinear classes and explain why such extrapolation result has been open for more than ten years.  For the sake of simplicity, let us consider the bilinear case with $\vec{p}=(1,1)$ so that $p=\frac12$. Having $\vec{w}=(w_1,w_2)\in A_{(1,1)}$ can be translated into $w_1^{\frac12}, w_2^{\frac12},w_1^{\frac12}w_2^{\frac12}\in A_1$ (see \cite{LOPTT} or Lemma \ref{lemma:main:II}). Hence, $\vec{w}=(w_1,w_2)\in A_{(1,1)}$, in contrast with $\vec{w}\in A_1\times A_1$, imposes less on each weight individually (since $w_i\in A_1$ easily implies $w_i^{\frac12}\in A_1$). But it incorporates a link between $w_1$ and $w_2$, which are no longer independent, since the product weight needs to satisfy $w_1^{\frac12}w_2^{\frac12}\in A_1$. For instance, we can take $w_1(x)=|x|^{-n}$ and $w_2(x)\equiv 1$ so that $\vec{w}=(w_1, w_2)\in A_{(1,1)}$, while $w_1\notin A_1$ since $w_1$ is not even locally integrable. On the other hand, once we pick $w_1(x)=|x|^{-n}$, there is a restriction on the possible weights $w_2$ for which $\vec{w}=(w_1, w_2)\in A_{(1,1)}$ since we would need to have that $|x|^{-\frac{n}2}w_2^{\frac12}\in A_1$ and this does not allow to take, for example, $w_2(x)=|x|^{-n}$. With these examples we can see already some of the difficulties that one encounters when trying to work with the multilinear classes of weights: first, one needs to work with component weights that are linked one another and, second, each individual weight might be non locally integrable but collectively the product should behave well. This might explain why any attempt to obtain a Rubio de Francia extrapolation theorem has been unsuccessful in the last years: the proofs for product weights in \cite{grafakos-martell04}, \cite{Duo}, \cite{CM-extrapol} treat each component independently and the conditions on the weights make them locally integrable, thus they do not naturally extend to the multilinear classes.  

In this paper we overcome these difficulties and obtain a multivariable Rubio de Francia extrapolation theorem which is not only valid for the classes $A_{\vec{p}}$ but also goes beyond and allows us to work with the classes $A_{\vec{p}, \vec{r}}$. The former classes are the natural ones for $\mathcal{M}$ and for multilinear Calder\'on-Zygmund operators, but the latter classes are related operators with restricted ranges of boundedness.  Indeed, these classes appeared in \cite{CPO2016} where weighted norm inequalities were obtained for the bilinear Hilbert transform in the case when the target space is Banach. One of the consequences of our main result is that extrapolation automatically extends these estimates to the case where the target spaces  are quasi-Banach.   Here we should emphasize that 
for the bilinear Hilbert transform (and some other classes of operators) the very robust helicoidal method developed in \cite{BM1, BM2, BM3, BM4} gives also estimates in the Banach and quasi-Banach case. Our result goes further and drops the dependence on the structural properties of the operator, once some weighted estimates are known for fixed exponents. This is indeed relevant as we can easily obtain vector-valued inequalities (some of them were proved in \cite{BM1, BM2, BM3, BM4} and \cite{CPO2016} for some particular classes of operators including the bilinear Hilbert transform).  To illustrate the applicability of our method we also prove weighted estimates for more singular operators  such as the commutators of multilinear operators with BMO functions. Very recently in \cite{BMMST} weighted norm inequalities were proved by using the  so-called Cauchy integral trick for operators satisfying weighted estimates associated within the class $A_{\vec{p}}$.  However, as this trick uses Minkowski's inequality, 
the estimates obtained in \cite{BMMST} were only valid in the Banach range.  As a result of our extrapolation result we can extend them to the quasi-Banach range.  As we are also able to work with the classes $A_{\vec{p}, \vec{r}}$, we can apply these ideas to prove new weighted estimates and vector-valued inequalities for the commutator of the  bilinear Hilbert transform with BMO functions. 

\medskip

In order to state our main result we need some notation. We  shall work on $\re^n$, $n\ge 1$, and by a
cube $Q$ in $\re^n$ we shall understand a cube with sides parallel to the coordinate axes.  Given a cube $Q$ and $f\in L^1_{\rm loc}(\re^n)$ we use the notation
\[
\dashint_Q f dx=\frac1{|Q|}\int_Q f dx.
\]
Hereafter, $m\ge 2$.  Given  $\vec p=(p_1,\dots, p_m)$ with $1\le
p_1,\dots, p_m<\infty$ and $\vec{r}=(r_1,\dots,r_{m+1})$ with $1\le r_1,\dots,r_{m+1}<\infty$,
we say that $\vec{r}\preceq \vec{p}$ whenever
\[
r_i\le p_i,
\quad
i=1,\dots, m;
\quad\mbox{and}
\quad r_{m+1}'>p,
\quad
\mbox{where}
\quad
\frac1p:=\frac1{p_1}+\dots+\frac1{p_{m}}
.
\]
Analogously, we say that
$\vec{r}\prec \vec{p}$ if $\vec{r}\preceq \vec{p}$ and moreover $r_i<p_i$ for every $i=1,\dots, m$. 
Notice that the fact that $\vec{r}\preceq\vec{p}$ forces that $\sum_{i=1}^{m+1}\frac1{r_i}>1$ and also $\frac1 p\le \sum_{i=1}^m \frac1{r_i}$. Hence, if $\sum_{i=1}^m \frac1{r_i}>1$ then we allow $p$ to be smaller than one. 

Under these  assumptions we can now introduce the classes of multilinear Muckenhoupt weights that we consider in the present paper, in Section \ref{section:multi-sparse} below we introduce some model operators whose weighted norm inequalities are governed by these classes. We say that   $\vec{w}=(w_1,\dots, w_m)\in A_{\vec p, \vec r}$, 
provided $0<w_i<\infty$ a.e. for every $i=1,\dots,m$ and 
\[
[\vec{w}]_{A_{\vec p, \vec r}}
=\sup_Q\Big(\dashint_Q w^{\frac {r_{m+1}'}{r_{m+1}'- p}}\d x\Big)^{\frac 1{p}-\frac
	1{r_{m+1}'}}
\prod_{i=1}^m \Big(\dashint_Q w_i^{\frac{r_i}{r_i-p_i}}\d x\Big)^{\frac 1{r_i}-\frac
	1{p_i}}<\infty,
\]
where $w=\prod_{i=1}^m w_i^{\frac{p}{p_i}} $.  When $r_{m+1}=1$ the term corresponding to $w$ needs to be replaced by $\Big(\dashint_Q w\d x\Big)^{\frac
	1{p}}$. Analogously, when $p_i=r_i$, the term
corresponding to $w_i$ needs to be replaced by $\esssup_Q w_i^{-\frac 1{p_i}}$. We note that $A_{\vec p, (1,\dots,1)}$ agrees with $A_{\vec p}$ introduced above.

We shall use the abstract
formalism of extrapolation families.    Hereafter $\mathcal{F}$ will
denote a family of $(m+1)$-tuples $(f,f_1,\ldots,f_m)$ of non-negative
measurable functions.   This approach to extrapolation has the
advantage that, for instance, vector-valued inequalities are an
immediate consequence of our extrapolation results.  We will discuss
applying this formalism to prove norm inequalities for specific operators
below.  For complete discussion of this
approach to extrapolation in the linear setting, see~\cite{CMP}.

Our main result is the following:

\begin{Theorem}\label{theor:extrapol-general}
	Let $ \mathcal F$ be a collection of $(m+1)$-tuples of non-negative functions. 	Assume that we have a vector $\vec{r}=(r_1,\dots,r_{m+1})$, with $1\le r_1,\dots,r_{m+1}<\infty$, and
	exponents $\vec p=(p_1,\dots, p_m)$, with $1\le
	p_1,\dots, p_m<\infty$ and $\vec{r}\preceq\vec{p}$,  such that given any $\vec w=(w_1,\dots, w_m) \in A_{\vec p,\vec{r}}$ the inequality
	\begin{equation}\label{extrapol:H}
	\|f\|_{L^{p}(w)} \le C([\vec w]_{A_{\vec p, \vec r}}) \prod_{i=1}^m\|f_i\|_{L^{p_i}(w_i)}
	\end{equation}
	holds for every $(f,f_1,\dots, f_m)\in \mathcal F$, where $\frac1p:=\frac1{p_1}+\dots+\frac1{p_m}$ and $w:=\prod_{i=1}^m w_i^{\frac{p}{p_i}}$.
	Then for all exponents $\vec q=(q_1,\dots,q_m)$, with $\vec{r}\prec\vec{q}$, and for all weights $\vec v=(v_1,\dots, v_m) \in A_{\vec q,\vec{r}}$ the inequality
	\begin{equation}\label{extrapol:C}
	\|f\|_{L^{q}(v)} \le C([\vec v]_{A_{\vec q, \vec r}})\prod_{i=1}^m \|f_i\|_{L^{q_i}(v_i)}
	\end{equation}
	holds for every $(f,f_1,\dots,f_m)\in \mathcal F$, where $\frac1q:=\frac1{q_1}+\dots+\frac1{q_m}$ and $v:=\prod_{i=1}^m v_i^{\frac{q}{q_i}}$.
	
	Moreover, for the same family of exponents and
	weights, and for all exponents $\vec{s}=(s_1,\dots, s_m)$ with $\vec{r}\prec\vec{s}$ 
		\begin{equation}\label{extrapol:vv}
	\bigg\|\Big(\sum_j (f^j)^s\Big)^\frac1s\bigg\|_{L^{q}(v)}
	\le
	C([\vec v]_{A_{\vec q, \vec r}})
	\prod_{i=1}^m\bigg\|\Big(\sum_j (f_i^j)^{s_i}\Big)^\frac1{s_i}\bigg\|_{L^{q_i}(v_i)}
	\end{equation}
	for all $\{(f^j, f_1^j, \dots, f_m^j)\}_j\subset\mathcal{F}$ and where $\frac1s:=\frac1{s_1}+\dots+\frac1{s_m}$.
\end{Theorem}

\bigskip

As a direct corollary of our main theorem, taking $\vec r=(1,\dots, 1)$, we provide the promised multivariable Rubio de Francia extrapolation theorem:
\begin{Corollary}\label{coro:ap}
	Let $ \mathcal F$ be a collection of $m+1$-tuples of non-negative functions. 	Assume that we have 
	exponents $\vec p=(p_1,\dots, p_m)$, with  $1\le p_1,\dots, p_m<\infty$, such that given any $\vec w=(w_1,\dots,w_m) \in A_{\vec p}$ the inequality
	\begin{equation}\label{extrapol:H:CZO}
	\|f\|_{L^{p}(w)} \le C([\vec w]_{A_{\vec p}}) \prod_{i=1}^m\|f_i\|_{L^{p_i}(w_i)}
	\end{equation}
	holds for every $(f,f_1,\dots, f_m)\in \mathcal F$, where $\frac1p:=\frac1{p_1}+\dots+\frac1{p_m}$ and $w:=\prod_{i=1}^m w_i^{\frac{p}{p_i}}$.
	Then for all exponents $\vec q=(q_1,\dots,q_m)$, with $1<q_1,\dots, q_m<\infty$, and for all weights $\vec v=(v_1,\dots,v_m) \in A_{\vec q}$ the inequality
	\begin{equation}\label{extrapol:C:ap}
	\|f\|_{L^{q}(v)} \le C([\vec v]_{A_{\vec q}})\prod_{i=1}^m \|f_i\|_{L^{q_i}(v_i)}
	\end{equation}
	holds for every $(f,f_1,\dots,f_m)\in \mathcal F$, where $\frac1q:=\frac1{q_1}+\dots+\frac1{q_m}$ and $v:=\prod_{i=1}^m v_i^{\frac{q}{q_i}}$.
	
	Moreover, for the same family of exponents and
	weights, and for all exponents $\vec{s}=(s_1,\dots, s_m)$ with $1<s_1,\dots, s_m<\infty$, 
	\begin{equation}\label{extrapol:vv:ap}
	\bigg\|\Big(\sum_j (f^j)^s\Big)^\frac1s\bigg\|_{L^{q}(v)}
	\le
	C([\vec v]_{A_{\vec q}})
	\prod_{i=1}^m\bigg\|\Big(\sum_j (f_i^j)^{s_i}\Big)^\frac1{s_i}\bigg\|_{L^{q_i}(v_i)}
	\end{equation}
	for all $\{(f^j, f_1^j, \dots, f_m^j)\}_j\subset\mathcal{F}$, where $\frac1s:=\frac1{s_1}+\dots+\frac1{s_m}$.
\end{Corollary}

\medskip

\begin{Remark}
As done in \cite[Section 6]{grafakos-martell04} one can formulate the previous results in terms of weak-type estimates. More precisely, in the context of Theorem \ref{theor:extrapol-general}, if in \eqref{extrapol:H} the left hand side term is replaced by $\|f\|_{L^{p,\infty}(w)} $ then in the conclusion we will have $\|f\|_{L^{q,\infty}(v)}$. The same occurs with Corollary \ref{coro:ap}. 
\end{Remark}

\medskip

\begin{Remark}
As discussed in \cite[Section 1]{CM-extrapol} one can easily get versions of the previous results where we make the a priori assumption that the left-hand sides of both our hypothesis and conclusion are finite. In certain applications this assumption is reasonable: for instance, when proving
Coifman-Fefferman type inequalities (cf. \cite{CMP}). The precise formulations and the proofs are left to the interested reader. 
\end{Remark}

\medskip

\begin{Remark}\label{remark:end-point}
	One can see that in Theorem \ref{theor:extrapol-general} if we start with $p_{i_0}=r_{i_0}$ for some given $i_0$ then in the conclusion we can relax $q_{i_0}>r_{i_0}$ to $q_{i_0}\ge r_{i_0}$, likewise, for Corollary  \ref{coro:ap}, if $p_{i_0}=1$ for some given $i_0$ then in the conclusion we can allow $q_{i_0}\ge 1$. To justify this, one just needs to apply the extrapolation procedure in all the other components which will eventually prove the case $q_{i_0}=p_{i_0}=r_{i_0}$. Apply finally the extrapolation on  the component $i_0$ to obtain the case $q_{i_0}>r_{i_0}$. Further details are left to the interested reader.
	\end{Remark}

\medskip

\medskip

\begin{Remark}\label{remark:iteration}
The formalism of extrapolation families is very useful to derive vector-valued inequalities. Indeed, as we will in Section \ref{section:proof:v-v}, from the main part of Theorem \ref{theor:extrapol-general} (that is, the fact that \eqref{extrapol:H} implies \eqref{extrapol:C}) one can easily obtain that \eqref{extrapol:vv} holds. This is done by extrapolation choosing an appropriate extrapolation family. This idea can be further exploited to obtain weighted  vector-valued inequalities  of the form 
\begin{equation}\label{vv-iterated}
\bigg\| \bigg(\sum_j 
\bigg(\sum_k (f^{jk})^s\bigg)^{\frac{t}{s}}
\bigg)^{\frac{1}{t}}
\bigg\|_{L^p(v)} 
\lesssim
\prod_{i=1}^m
\bigg\| \bigg(\sum_j 
\bigg(\sum_k (f_i^{jk})^{s_i}\bigg)^{\frac{t_i}{s_i}}
\bigg)^{\frac{1}{t_i}}
\bigg\|_{L^{q_i}(v_i)}. 
\end{equation}
The argument to show this is in Section  \ref{section:iterarion} and it is straightforward to see that one can repeat this procedure to obtain iterated vector-valued inequalities with arbitrary number of ``sums''. The precise statements are left to the interested the reader. 
\end{Remark}

To conclude with this introduction let us briefly present one of the novel ideas that we introduce to obtain our extrapolation result. As mentioned before, the conditions $A_{\vec{p}}$ or $A_{\vec{p},\vec{r}}$ contain  implicitly some link between the different components of the vector weight. In  Lemma \ref{lemma:main} we present a structural result for the classes $A_{\vec{p},\vec{r}}$ which makes this connection explicit. Although the result is quite technical, in the bilinear case and for the class $A_{(1,1)}$ it has a very simple and illustrative statement: $\vec{w}=(w_1,w_2)\in A_{(1,1)}$ if and only if $w_1^{\frac12}\in A_1$ and $w_2^{\frac12}\in A_1(w_1^{\frac12})$ (the latter means that $w_2^{\frac12}$ satisfies an $A_1$ condition with respect to the underlying measure $w_1^{\frac12}$). This equivalence allows us, among other things, to easily construct vector weights in $A_{(1,1)}$:  we simply use that  $A_1$ weights are essentially a Hardy-Littlewood maximal function raised to a power strictly smaller than $1$.  On the other hand, we highlight that this is one of the key ideas that have allowed us to obtain our Rubio de Francia extrapolation for multilinear Muckenhoupt classes: roughly speaking we reduce matters to some off-diagonal extrapolation in one component with respect to some fixed doubling measure (in the previous example this would be $w_1^{\frac12}$) and then iterate this procedure for the other components.

The plan of the paper is as follows. In the following section we present some immediate applications of our extrapolation results. We easily reprove some vector-valued inequalities for bilinear (for the sake of specificity) Calderón-Zygmund operators. Next we look into some sparse domination formulas and how these easily give estimates in the Banach case (by a direct computation) and in the quasi-Banach case by extrapolation. From these we automatically obtain vector-valued inequalities. In Section \ref{section:BHT} we pay special attention to the important case of the bilinear Hilbert transform and explain how extrapolation  easily produces a plethora of estimates starting from the weighted norm inequalities proved in \cite{CPO2016}. The study of the commutators with BMO functions is in Section \ref{section:comm} where we obtain estimates that are new for the bilinear Hilbert transform. In Section \ref{section:aux} we give some auxiliary results including Lemma \ref{lemma:main}  which gives the characterization of the class $A_{\vec p, \vec r}$ mentioned above. The proof of our main result and the estimates for the commutators are given respectively in Sections \ref{section:proof-main}  and \ref{section:proof-comm}.

\section{Applications}

Here we present some applications. In some cases we give elementary proofs of some known estimates but in other we prove new estimates. 

\subsection{Multilinear Calderón-Zygmund operators}

For the sake of conciseness let us just handle the bilinear case. We recall the definition of the bi-sublinear Hardy-Littlewood maximal function 
\[
\mathcal{M}(f,g)(x)=\sup_{Q\ni x} \Big(\dashint_Q |f(y)|dy\Big)\Big(\dashint_Q |g(y)|dy\Big).
\]

Given a bilinear operator $T$ a priori defined from $\mathcal S\times\mathcal S$ into $\mathcal S'$ of the form
\[
T(f,g)(x)=\int_{\mathbb R^n}\int_{\mathbb R^n} K(x, y, z)f(y)g(z)\,\d y\d z
\]
we say that $T$ is a bilinear Calderón-Zygmund operator if it can be extended as a bounded operator from $L^{p_1}\times L^{p_2}$ to $L^p$ for some $1<p_1, p_2<\infty$ with $1/p_1+1/p_2=1/p$, and its distributional kernel $K$ coincides, away from the diagonal $\{(x,y,z) \in \mathbb R^{3n}: x=y=z \}$, with a function  $K(x,y,z)$ locally integrable which satisfies estimates of the form
\[
|\partial^\alpha K(x,y,z)| \lesssim \big(|x-y| + |x-z| + |y-z|\big)^{-2n-|\alpha|}, |\alpha|\le 1.
\]
The estimates on $K$ above are not the most general that one can impose in such theory, see \cite{GT}. Bilinear Calderón-Zygmund operators and $\mathcal{M}$ are known to satisfy weighted norm inequalities for the classes $A_{\vec{p}}=A_{\vec{p},(1,1,1)}$, see \cite{LOPTT}. As a consequence of Theorem \ref{theor:extrapol-general} we easily obtain the following vector-valued inequalities:

\begin{Corollary}\label{corol:CZO}
	Let $T$ be a bilinear Calderón-Zygmund operator. For every $\vec{p}=(p_1,p_2)$, $\vec{s}=(s_1,s_2)$ with $1<p_1,p_2,s_1,s_2<\infty$ and for every $\vec{w}=(w_1,w_2)\in A_{\vec{p}} $ one has
	\[
	\bigg\|\Big(\sum_j \mathcal{M}(f_j,g_j)^s\Big)^\frac1s\bigg\|_{L^{p}(w)}
	\lesssim
	\bigg\|\Big(\sum_j |f_j|^{s_1}\Big)^\frac1{s_1}\bigg\|_{L^{p_1}(w_1)}
	\bigg\|\Big(\sum_j |g_j|^{s_2}\Big)^\frac1{s_2}\bigg\|_{L^{p_2}(w_2)}
	\]
	and
	\[
	\bigg\|\Big(\sum_j |T(f_j,g_j)|^s\Big)^\frac1s\bigg\|_{L^{p}(w)}
	\lesssim
	\bigg\|\Big(\sum_j |f_j|^{s_1}\Big)^\frac1{s_1}\bigg\|_{L^{p_1}(w_1)}
	\bigg\|\Big(\sum_j |g_j|^{s_2}\Big)^\frac1{s_2}\bigg\|_{L^{p_2}(w_2)},
	\]
	where $\frac1p=\frac1{p_1}+\frac1{p_2}$, $\frac1s=\frac1{s_1}+\frac1{s_2}$, and  $w=w_1^{\frac{p}{p_1}}w_2^{\frac{p}{p_1}}$.
\end{Corollary}

 Notice that as explained in Remark \ref{remark:iteration} one can easily obtain iterated weighted vector-valued inequalities, the precise statement is left to the interested reader. 

We can also use extrapolation to prove Marcinkiewicz-Zygmund inequalities for multilinear Calderón-Zygmund operators (here as before we just present the bilinear case).   Very recently Carando \textit{et al.} \cite{Carando:2016vm} extended some result from \cite{grafakos-martell04} and \cite{BPV} by proving the following weighted
Marcinkiewicz-Zygmund inequalities. Let $T$ be a bilinear Calderón-Zygmund operator. Let $1<r\le 2$ and let $1<q_1,q_2<\infty$ if $r=2$ or $1<q_1,q_2<r$ if $1<r<2$. Then for $\vec{w}=(w_1,w_2)\in A_{\vec{q}}$  there holds
	\begin{equation} \label{eqn:CZ-MZ0}
	\bigg\| \bigg(\sum_{i,j } 
	|T(f_{i},g_{j})|^r\bigg)^{\frac{1}{r}}
	\bigg\|_{L^q(w)} 
	\leq 
	\bigg\| \bigg( \sum_{i} |f_{i}|^r\bigg)^{\frac{1}{r}}\bigg\|_{L^{q_1}(w_1)}
	\bigg\| \bigg( \sum_{j} |g_{j}|^r\bigg)^{\frac{1}{r}}
	\bigg\|_{L^{q_2}(w_2)}, 
	\end{equation}
	where $\frac{1}{q}=\frac{1}{q_1}+\frac{1}{q_2}$ and $w=w_1^{\frac q{q_1}}w_2^{\frac q{q_2}}$.
	
By using extrapolation we can remove the restriction $q_1,q_2<r$ when $1<r<2$ (for a version of the following result in the context of product of Muckenhoupt classes the reader is referred to \cite{CM-extrapol}).

\begin{Corollary} \label{coro:CZ-MZ-new}
Let $T$ be bilinear Calderón-Zygmund operator. Given $1<r\le 2$ and $1<q_1,q_2<\infty$,  then
\eqref{eqn:CZ-MZ0} holds for all $\vec{w}=(w_1,w_2)\in A_{\vec{q}}$.
\end{Corollary}

\subsection{Multilinear sparse forms:  bilinear Calderón-Zygmund operators}

In the previous section we derive vector-valued estimates for bilinear Calderón-Zyg\-mund operators as a consequence of the theory developed in \cite{LOPTT}. Here we 
would like to show that these can be easily obtained from a particular choice of $\vec{p}$ with the help of a certain sparse domination which will also motivate the definition of more general multilinear sparse operators.  We start with an estimate proved independently and simultaneously in \cite{conde-rey} and \cite{lerner-nazarov}: for any $f,g\in C^\infty_c(\re^n)$, one has
\begin{equation}\label{sparsedom-BCZO}
|T(f,g)|\lesssim
\sum_{i=1}^{3^n} T_{\mathcal{S}_i}(f,g),
\end{equation}
where for each $i$, $\mathcal{S}_i=\{Q\}\subset\mathbb{D}_i$ (here $\mathbb{D}_i$ is a dyadic grid) is a sparse family with sparsity constant $\frac12$ and
\[
T_{\mathcal{S}_i}(f,g)
=
\sum_{Q\in\mathcal{S}_i} \Big(\dashint_Q |f|dx\Big) \Big(\dashint_Q |g|dx\Big) \chi_Q.
\]	
Let us recall that $\mathcal{S}=\{Q\}$ is sparse family with constant $\zeta\in (0,1)$ if for every  $Q\in\mathcal{S}$ there exists $E_Q\subset Q$ such that $|E_Q|>\zeta |Q|$ and the sets $\{E_Q\}_{Q\in \mathcal{S}}$ are pairwise disjoint. 

To proceed we follow an argument in \cite{DLP}, see also \cite{LMS}, which in turn is a bilinear extension of the linear case proof in \cite{CMP-AIM}. We pick the ``natural'' exponents $p_1=p_2=3$ and $p=\frac32$, and take $\vec{w}=(w_1,w_2)\in A_{(3,3)}$, let $w=w_1^{\frac12}w_2^{\frac12}$, and write $\sigma_1=w_1^{1-p_1'}=w_1^{-\frac12}$, $\sigma_2=w_2^{1-p_2'}=w_2^{-\frac12}$. Without loss of generality we may assume that $f, g\ge 0$ and use duality to see that there exists $0\le h\in L^{3}(w^{-2})$ with $\|h\|_{L^{3}(w^{-2})}=1$ such that
\begin{equation}\label{BCZO}
\|T(f,g)\|_{L^{\frac32}(w)}
=
\int_{\re^n} h|T(f,g)|dx
\lesssim
\sup_{\mathcal{S}}
\int_{\re^n} h T_{\mathcal{S}}(f,g)dx
=
\sup_{\mathcal{S}} \Lambda_{\mathcal{S}} (f,g,h),
\end{equation}
where we have used \eqref{sparsedom-BCZO}, 
\begin{equation}
\Lambda_{\mathcal{S}} (f,g,h):=
\sum_{Q\in\mathcal{S}}|Q|\Big(\dashint_Q h dx\Big)\Big(\dashint_Q fdx\Big) \Big(\dashint_Q gdx\Big), 
\end{equation}
and the sup runs over all sparse collections $\mathcal{S}$ with sparsity constant $\frac12$. To continue with our estimate we just need to estimate an arbitrary $\Lambda_{\mathcal{S}}$:
\begin{align}\label{arfer}
&\Lambda_{\mathcal{S}} (f,g,h)
=
\sum_{Q\in\mathcal{S}}|Q|
\Big(\dashint_Q h w^{-1} dw\Big) 
\Big(\dashint_Q f \sigma_1^{-1} d\sigma_1\Big) 
\Big(\dashint_Q g \sigma_2^{-1} d\sigma_2\Big) 
\\
&
\hskip3cm\times
\Big(\dashint_Q wdx\Big)
\Big(\dashint_Q \sigma_1 dx\Big) \Big(\dashint_Q \sigma_2dx\Big) 
\\ \nonumber
&
\quad\le
2[\vec{w}]_{A_{(3,3)}}^{\frac32}
\sum_{Q\in\mathcal{S}} 
|E_Q|
\Big(\dashint_Q h w^{-1} dw\Big) 
\Big(\dashint_Q f \sigma_1^{-1} d\sigma_1\Big) 
\Big(\dashint_Q g \sigma_2^{-1} d\sigma_2\Big)
\\ \nonumber
&
\quad\le
2[\vec{w}]_{A_{(3,3)}}^{\frac32}
\int_{\re^n} 
M_{w}^{\mathbb{D}}(h w^{-1})
M_{\sigma_1}^{\mathbb{D}}(f\sigma_1^{-1}) M_{\sigma_2}^{\mathbb{D}}(g\sigma_2^{-1})
\\ \nonumber
&
\quad\le
2[\vec{w}]_{A_{(3,3)}}^{\frac32}
\|M_{w}^{\mathbb{D}}(hw^{-1})\|_{L^3(w)}
\|M_{\sigma_1}^{\mathbb{D}}(f\sigma_1^{-1})\|_{L^3(\sigma_1)}
\|M_{\sigma_2}^{\mathbb{D}}(g\sigma_2^{-1})\|_{L^3(\sigma_2)}
\\ \nonumber
&
\quad\le
\frac{27}{4}
[\vec{w}]_{A_{(3,3)}}^{\frac32}
\|hw^{-1}\|_{L^3(w)}
\|f\sigma_1^{-1}\|_{L^3(\sigma_1)}
\|g\sigma_2^{-1}\|_{L^3(\sigma_2)}
\\ \nonumber
&
\quad
=
\frac{27}{4}
[\vec{w}]_{A_{(3,3)}}^{\frac32}
\|f\|_{L^3(w_1)}
\|g\|_{L^3(w_2)},
\end{align}
where we have used  that $\mathcal{S}$ is a sparse family with constant $\frac12$, hence the sets $\{E_Q\}_{Q\in \mathcal{S}}$ are pairwise disjoint; H\"older's inequality;  and finally that $M_{\mu}^{\mathbb{D}}$, the dyadic maximal operator associated with the dyadic grid $\mathbb{D}$ and with underlying measure $\mu$, is bounded  on $L^3(\mu)$ with bound $3'=\frac32$ (see for instance \cite[Lemma 2.3]{CMP-AIM}). Collecting the obtained estimates we therefore conclude that $T:L^3(w_1)\times L^3(w_2)\to L^{\frac32}(w)$ for every $\vec{w}=(w_1,w_2)\in A_{(3,3)}$ and where $w=w_1^{\frac12}w_2^{\frac12}$. Using Corollary  \ref{coro:ap} with $\vec{p}=(3,3)$ and the  family $\mathcal{F}$ consisting in the collection of $3$-tuples  $(|T(f,g)|,|f|,|g|)$ with $f,g\in C^\infty_c(\re^n)$ we easily conclude that $T:L^{p_1}(w_1)\times L^{p_2}(w_2)\to L^p(w)$ for every $1<p_1,p_2<\infty$ where $\frac1p=\frac1{p_1}+\frac1{p_1}$, $\vec{w}=(w_1,w_2)\in A_{(p_1,p_2)}$ and $w=w_1^{\frac{p}{p_1}}w_2^{\frac{p}{p_2}}$. Also, we automatically get  the corresponding the vector-valued inequalities.

\subsection{Multilinear sparse forms: the general case}\label{section:multi-sparse}

In this section we present some multilinear sparse forms whose weighted norm inequalities are governed by the class $A_{\vec{p},\vec{r}}$. Here we would like to emphasize that the natural argument based on duality gives estimates in the Banach range, and extrapolation allows us to extend them to the case on which the target space is quasi-Banach. Let us introduce the sparse forms. Given a dyadic grid $\mathbb{D}$, a sparse family $\mathcal{S}\subset \mathbb{D}$, and
$\vec{r}=(r_1,\dots, r_{m+1})$ with $r_i\ge 1$, for every $1\le i\le m+1$, and $\frac1{r_1}+\dots+\frac1{r_{m+1}}>1$, let us define
\[
\Lambda_{\mathcal{S},\vec{r}} (f_1, \dots, f_m, h)=\sum_{Q\in \mathcal S} |Q|\Big(\dashint_Q
|h|^{r_{m+1}}dx\Big)^{\frac 1{r_{m+1}}}\prod_{i=1}^m\Big(\dashint_Q |f_i|^{r_i}dx\Big)^{\frac
	1{r_i}}.
\]
Our goal is to present a general framework to establish weighted estimates for operators which are controlled by sparse forms $\Lambda_{\mathcal{S},\vec{r}}$. In that case we only need to establish the corresponding estimates for  $\Lambda_{\mathcal{S},\vec{r}}$ and this is a quite easy task.

Fix $\vec{r}=(r_1,\dots,r_{m+1})$, with  $r_i\ge 1$ for $1\le i\le m+1$ and a sparsity constant $\zeta\in(0,1)$. Consider an operator $T$ (we do not need any linearity or sublinearity) and we seek to show that  $T:L^{p_1}(w_1)\times\dots\times L^{p_m}(w_m)\to L^p(w)$
for $\vec{p}$ in some range and where as usual $w=\prod_{i=1}^m w_i^{\frac p{p_i}}$. By duality, and provided that $p>1$, we can find $0\le h\in L^{p'}(w^{1-p'}) $ with $\|h\|_{L^{p'}(w^{1-p'})}=1$  so that
\[
\|T (f_1, \dots, f_m)\|_{L^p(w)}
=
\int_{\re^n} h|T (f_1, \dots, f_m)|dx.
\]
Our main assumption is that 
\begin{equation}\label{T-dom-form}
\|T (f_1, \dots, f_m)\|_{L^p(w)}
=
\int_{\re^n} h|T (f_1, \dots, f_m)|dx
\lesssim 
\sup_{\mathcal{S}}\Lambda_{\mathcal{S},\vec{r}}(f_1,\dots,f_m,h),
\end{equation}
where the sup runs over all sparse families with sparsity constant $\zeta$.  

We next present some operators satisfying the previous assumption. First, if $T$ is a bilinear Calderón-Zygmund operator as before, then \eqref{BCZO} shows that \eqref{T-dom-form} holds with $\vec{r}=(1,1,1)$ and the sparsity constant is $\zeta=\frac12$ ---the same occurs with multilinear Calderón-Zygmund operators 
with $\vec{r}=(1,\dots,1)$ (see \cite{conde-rey} or \cite{lerner-nazarov}). 
The second example is a class of rough bilinear singular integrals studied by Barron for which 
\eqref{T-dom-form} holds for $\vec{r}=(r_1,r_2,r_3)$ with any  $1<r_1,r_2,r_3<\infty$ (see  \cite[Therorem 1]{barron}). The last and the most prominent example is that of the bilinear Hilbert transform defined as
\[
BH(f, g)(x)=\text{p.v}\,\int_{\mathbb R}f(x-t)g(x+t)\frac{dt}{t}.
\]
It is a bilinear operator whose multiplier, unlike the ones for bilinear Calderón-Zygmund operators which are singular only at the origin, is singular along a line when viewed in the frequency plane.  Lacey and Thiele \cite{LT1, LT2} (see also \cite{GL}) showed that $BH$ maps $L^{p_1}\times L^{p_2}\to L^p$ for $1<p_1,p_2\le\infty$ and $\frac1p=\frac1{p_1}+\frac1{p_2}<\frac32$.   In \cite[Theorem 2]{CPO2016} (see also \cite{BM3}), this operator and some other bilinear multipliers have been shown to satisfy \eqref{T-dom-form}
with $\vec{r}=(r_1,r_2,r_3)$ satisfying $1<r_1,r_2,r_3<\infty$ and
\begin{equation}\label{cond-adm-mot}
\frac1{\min\{r_1,2\}}+\frac1{\min\{r_2,2\}}+\frac1{\min\{r_3,2\}}<2.
\end{equation}

\medskip

Continuing with our argument, for all the previous examples we are going to see that \eqref{T-dom-form} allows us to reduce the weighted norm inequalities of $T$ to those of the sparse forms $\Lambda_{\mathcal{S},\vec{r}}$. However, this eventually produces estimates in $L^p(w)$ where $p> 1$ and this is where extrapolation is relevant, since it permits us to  easily remove such restriction. As done before for the bilinear Calderón-Zygmund operators we are going to obtain estimates for some particular choice of $\vec{p}$ (let us note that in the previous references ---except for \cite{BM3}--- more general $\vec{p}$ are considered but the constrain $\frac1p<1$ is always present). 
Take $\vec{p}=(p_1,\dots,p_m)$ where
\[
p_i=\frac{r_i}{r},
\quad 1\le i\le m,
\qquad\mbox{and}\qquad
\frac1r=\sum_{i=1}^{m+1} \frac1{r_i}>1,
\]
and let
\[
\frac1{p}
=
r\sum_{i=1}^{m} \frac1{r_i}
=
1-\frac{r}{r_{m+1}}
=
\frac1{(\frac{r_{m+1}}{r})'}.
\]
Note that $\vec{r}\prec\vec{p}$. Fix $\vec{w}\in A_{\vec{p},\vec{r}}$ and let $w=\prod_{i=1}^m w_i^{\frac p{p_i}}$. We set
\begin{equation}\label{defi-dualweights}
\sigma_{m+1}=w^{\frac {r_{m+1}'}{r_{m+1}'- p}}
=
w^{(p'-1)\frac{r}{1-r}};
\qquad
\sigma_i=w_i^{\frac{r_i}{r_i-p_i}}=w_i^{-\frac{r}{1-r}}, \quad 1\le i\le m,
\end{equation}
hence
\begin{equation}\label{obs:Apr-case}
[w]_{A_{\vec{p},\vec{r}}}^{\frac1{1-r}}
=
\sup_Q\Big(\dashint_Q w^{\frac {r_{m+1}'}{r_{m+1}'- p}}\d x\Big)^{\frac
	1{r_{m+1}}}
\prod_{i=1}^m \Big(\dashint_Q w_i^{\frac{r_i}{r_i-p_i}}\d x\Big)^{\frac 1{r_i}}
=
\sup_Q
\prod_{i=1}^{m+1} \Big(\dashint_Q \sigma_i\d x\Big)^{\frac 1{r_i}}
\end{equation}
and
\begin{equation}\label{prod-weights1}
\prod_{i=1}^{m+1}\sigma_i^{\frac{r}{r_i}}
=
w^{\frac{r}{p(1-r)}}
\prod_{i=1}^{m}w_i^{-\frac{r}{p_i(1-r)}}
=
1
\end{equation}

Assume next that $f_1,\dots,f_m, h\ge 0$ and set $f_{m+1}=h$. Given a sparse family $\mathcal{S}\subset\mathbb{D}$ with sparsity constants $\zeta$,  we proceed as in \eqref{arfer} to obtain
\begin{align}\label{wgeagste}
&\Lambda_{\mathcal{S},\vec{r}} (f_1,\dots, f_m ,h)
=
\sum_{Q\in \mathcal S} |Q|
\prod_{i=1}^{m+1}\Big(\dashint_Q f_i^{r_i}dx\Big)^{\frac
	1{r_i}}
\\  \nonumber
&\qquad=
\sum_{Q\in\mathcal{S}}|Q| \prod_{i=1}^{m+1}\Big(\dashint_Q f_i^{r_i}\sigma_i^{-1} d\sigma_i\Big)^{\frac
	1{r_i}}\prod_{i=1}^{m+1} \Big(\dashint_Q \sigma_i\d x\Big)^{\frac 1{r_i}}
\\ \nonumber
&
\qquad\le
\zeta^{-1} [\vec{w}]_{A_{\vec{p},\vec{r}}}^{\frac1{1-r}}
\sum_{Q\in\mathcal{S}} 
|E_Q|
\prod_{i=1}^{m+1}\Big(\dashint_Q f_i^{r_i}\sigma_i^{-1} d\sigma_i\Big)^{\frac
	1{r_i}}
\\ \nonumber
&
\qquad\le
\zeta^{-1}[\vec{w}]_{A_{\vec{p},\vec{r}}}^{\frac1{1-r}}
\int_{\re^n} 
\prod_{i=1}^{m+1}M_{\sigma_i}^{\mathbb{D}}(f_i^{r_i}\sigma_i^{-1})^\frac1{r_i} dx
\\ \nonumber
&
\qquad =
\zeta^{-1} [\vec{w}]_{A_{\vec{p},\vec{r}}}^{\frac1{1-r}}
\int_{\re^n} 
\prod_{i=1}^{m+1}M_{\sigma_i}^{\mathbb{D}}(f_i^{r_i}\sigma_i^{-1})^\frac1{r_i} \sigma_i^\frac{r}{r_i} dx
\\ \nonumber
&
\qquad\le
\zeta^{-1} [\vec{w}]_{A_{\vec{p},\vec{r}}}^{\frac1{1-r}}
\prod_{i=1}^{m+1}\big\|M_{\sigma_i}^{\mathbb{D}}(f_i^{r_i}\sigma_i^{-1})\big\|_{L^{\frac1{r}}(\sigma_i)}^{\frac1{r_i}}
\\ \nonumber
&
\qquad\le
\zeta^{-1} (1-r)^{-(m+1)}[\vec{w}]_{A_{\vec{p},\vec{r}}}^{\frac1{1-r}}
\prod_{i=1}^{m+1}\big\|f_i^{r_i}\sigma_i^{-1}\|_{L^{\frac1{r}}(\sigma_i)}^{\frac1{r_i}}
\\
&
\qquad \nonumber
=
\zeta^{-1} (1-r)^{-(m+1)}[\vec{w}]_{A_{\vec{p},\vec{r}}}^{\frac1{1-r}}
\|h\|_{L^{p'}(w^{1-p'})}
\prod_{i=1}^{m+1}\|f_i\|_{L^{p_i}(w_i)},
\end{align}
where we have used  that $\mathcal{S}$ is a sparse family with constant $\zeta$,  hence the sets $\{E_Q\}_{Q\in \mathcal{S}}$ are pairwise disjoint; \eqref{prod-weights1}; H\"older's inequality along with $\sum_{i=1}^m \frac{r}{r_i}=1$; that $M_{\mu}^{\mathbb{D}}$ (the dyadic maximal operator associated with the dyadic grid $\mathbb{D}$ and with underlying measure $\mu$) is bounded  on $L^{\frac1r}(\mu)$ with bound $(1-r)^{-1}$ since $r<1$ (see for instance \cite[Lemma 2.3]{CMP-AIM}); and finally \eqref{defi-dualweights}. If we now plug the obtained inequality in \eqref{T-dom-form} and use that $\|h\|_{L^{p'}(w^{1-p'})}=1$ we conclude as desired that
\begin{equation}\label{staring-form}
\|T (f_1, \dots, f_m)\|_{L^p(w)}
\le 
\zeta^{-1} (1-r)^{-(m+1)}[\vec{w}]_{A_{\vec{p},\vec{r}}}^{\frac1{1-r}}
\prod_{i=1}^{m+1}\|f_i\|_{L^{p_i}(w_i)},
\end{equation}
for all   $\vec{w}\in A_{\vec{p},\vec{r}}$ where $w=\prod_{i=1}^m w_i^{\frac p{p_i}}$.

\begin{Remark}
	Notice that we have shown that $\vec{w}\in A_{\vec{p},\vec{r}}$ is sufficient for \eqref{wgeagste}. Furthermore,  it can be seen that it is also necessary. Indeed if we just take $\mathcal{S}$ consisting on a single arbitrary cube $Q$ and we let $f_i=\sigma_i^{\frac{1}{r_i}}\chi_Q$ then
	\begin{multline*}
	|Q|
	\prod_{i=1}^{m+1}\Big(\dashint_Q \sigma_i dx\Big)^{\frac
		1{r_i}}
	=
	\Lambda_{\mathcal{S},\vec{r}} (f_1,\dots, f_m ,f_{m+1})
	\\
	\le
	C_0
	\|f_{m+1}\|_{L^{p'}(w^{1-p'})}
	\prod_{i=1}^{m+1}\|f_i\|_{L^{p_i}(w_i)}
	=
	C_0 \prod_{i=1}^{m+1}\big\|f_i^{r_i}\sigma_i^{-1}\|_{L^{\frac1{r}}(\sigma_i)}^{\frac1{r_i}}
	=
	C_0 |Q|
	\end{multline*}
	which eventually leads to
	\[
	\prod_{i=1}^{m+1}\Big(\dashint_Q \sigma_i dx\Big)^{\frac
		1{r_i}}
	\le C_0.
	\]
	Taking the sup over all cubes and using \eqref{obs:Apr-case}  we immediately see that $\vec{w}\in A_{\vec{p},\vec{r}}$ with $[w]_{A_{\vec{p},\vec{r}}}\le C_0^{1-r}$.
\end{Remark}

If we now use \eqref{staring-form} as a starting estimate, Theorem \ref{theor:extrapol-general} immediately gives the following result: 
\begin{Corollary}\label{corol:sparese-general}
	Fix $\vec{r}=(r_1,\dots,r_{m+1})$, with  $r_i\ge 1$ for $1\le i\le m+1$, and $\sum_{i=1}^{m+1}\frac 1{r_i}>1$, and a sparsity constant $\zeta\in(0,1)$.
	Let $T$ be an operator 	so that for every $f_1,\dots ,f_m,h\in C_c^\infty(\re^n)$ 
\begin{equation}\label{sparse-domination}
	\int_{\re^n} h|T (f_1, \dots, f_m)|dx
	\lesssim 
	\sup_{\mathcal{S}}\Lambda_{\mathcal{S},\vec{r}}(f_1,\dots,f_m,h),
\end{equation}
	where the sup runs over all sparse families with sparsity constant $\zeta$. Then for all exponents $\vec q=(q_1,\dots,q_m)$, with $\vec{r}\prec\vec{q}$, for all weights $\vec v=(v_1,\dots,v_m) \in A_{\vec q,\vec{r}}$,  and for all $f_1,\dots, f_m\in C_c^\infty(\re^n)$
	\begin{equation}\label{extrapol:C:CZO}
	\|T(f_1,\dots, f_m)\|_{L^{q}(v)} \lesssim \prod_{i=1}^m \|f_i\|_{L^{q_i}(v_i)},
	\end{equation}
	where $\frac1q:=\frac1{q_1}+\dots+\frac1{q_m}$ and $v:=\prod_{i=1}^mv_i^{\frac{q}{q_i}}$.
	Moreover, for the same family of exponents and
	weights, and for all exponents $\vec{s}=(s_1,\dots, s_m)$ with $\vec{r}\prec\vec{s}$ 
	\begin{equation}\label{extrapol:vv:CZO}
	\bigg\|\Big(\sum_j |T(f_1^j,\dots, f_m^j)|^s\Big)^\frac1s\bigg\|_{L^{q}(v)}
	\lesssim
	\prod_{i=1}^m\bigg\|\Big(\sum_j |f_i^j|^{s_i}\Big)^\frac1{s_i}\bigg\|_{L^{q_i}(v_i)},
	\end{equation}
	for all $f_1^j,\dots, f_m^j\in C_c^\infty(\re^n)$ and where $\frac1s:=\frac1{s_1}+\dots+\frac1{s_m}$.
\end{Corollary}

 We would like to observe that iterated weighted vector-valued inequalities can be obtained from Remark \ref{remark:iteration}, details are left to the interested reader.

\subsection{The bilinear Hilbert transform}\label{section:BHT}

In this section we establish weighted norm inequalities and vector-valued inequalities for the bilinear Hilbert transform. As discussed in the previous section this operator fits into Corollary \ref{corol:sparese-general} with $m=2$ and with $\vec{r}=(r_1,r_2,r_3)$ satisfying \eqref{cond-adm-mot}. As a matter of fact  it was shown in 
\cite[Theorem 3]{CPO2016} (see also \cite{BM3}) that 
\begin{equation}\label{BHT-dom-form}
\int_{\re^n} h|BH(f,g))|dx
\lesssim 
\sup_{\mathcal{S}}\Lambda_{\mathcal{S},\vec{r}}(f,g,h),
\end{equation}
where the sup runs over all sparse families with sparsity constant $\frac16$. As a consequence of this, it is obtained in \cite[Corollary 4]{CPO2016} that if $\vec{r}\prec\vec{p}$ with $p>1$ and $\vec{w}= (w_1, w_2)\in A_{\vec{p},\vec{r}}$ 
then
\begin{equation}\label{diP}
BH: L^{p_1}(w_1)\times L^{p_2}(w_2)\to L^p(w)	
\end{equation}	
with $w=w_1^{\frac{p}{p_1}}w_2^{\frac{p}{p_2}}$.
 As an immediate consequence of Theorem \ref{theor:extrapol-general} (or of Corollary \ref{corol:sparese-general}) we can get estimates for $p\le 1$ (reproving some of the estimates in \cite{BM3}). This extends the recent results in \cite{CM-extrapol} where the case of product of $A_p$ classes was obtained by extrapolation.

On the other hand, as a corollary we can prove vector-valued weighted norm inequalities which extend \cite{BM1, BM2} (see also \cite{CPO2016} when $p>1$ and \cite{HL}, \cite{Sil} for earlier results)
where the unweighted case was considered and reprove some of the estimates in \cite{BM3, BM4} (when all the exponents are finite). Also, we go beyond \cite[Section 5]{CM-extrapol} where weighted estimates were derived for product of $A_p$ classes. Here it is important to emphasize that \cite{CM-extrapol} did not recover the full range of vector-valued estimates from \cite{BM1, BM2} in the unweighted situation. The problem there is that extrapolation is done for  product of $A_p$ classes and this adds some unavoidable restriction in the exponents. Our extrapolation result is able to fine-tune and remove that restriction, this occurs since we work with more general classes of weights. Let us also note that our extrapolation allows us to obtain weighted estimates (and also vector-valued inequalities) in the quasi-Banach range (i.e., $p<1$) as a result of the estimates in the Banach case (i.e., $p\ge 1$).

\begin{Corollary}\label{corol:BH-1st}
	Let $\vec{r}=(r_1,r_2,r_3)$ be such that $1<r_1,r_2,r_3<\infty$ and
	\begin{equation}\label{cond-adm:corol}
	\frac1{\min\{r_1,2\}}+\frac1{\min\{r_2,2\}}+\frac1{\min\{r_3,2\}}<2.
	\end{equation}
	Let $\vec{p}=(p_1,p_2)$, $\vec{s}=(s_1,s_2)$  with $1<p_1, p_2, s_1,s_2<\infty$, and set  $\frac1p:=\frac1{p_1}+\frac1{p_2}$,  $\frac1s:=\frac1{s_1}+\frac1{s_2}$.
	If $\vec{r}\prec\vec{p}$ and $\vec{w}= (w_1, w_2)\in A_{\vec{p},\vec{r}}$  then
	\[
	BH: L^{p_1}(w_1)\times L^{p_2}(w_2)\to L^p(w),
	\]
	where $w=w_1^{\frac{p}{p_1}}w_2^{\frac{p}{p_2}}$. Moreover, if additionally  $\vec{r}\prec\vec{s}$ then $BH: L^{p_1}_{\ell^{s_1}}(w_1)\times L^{p_2}_{\ell^{s_2}}(w_2)\to L_{\ell^s}^p(w)$, that is,
	\[
	\bigg\|\Big(\sum_j |BH(f_j,g_j)|^s\Big)^\frac1s\bigg\|_{L^{p}(w)}
	\lesssim
	\bigg\|\Big(\sum_j |f_j|^{s_1}\Big)^\frac1{s_1}\bigg\|_{L^{p_1}(w_1)}
	\bigg\|\Big(\sum_j |g_j|^{s_2}\Big)^\frac1{s_2}\bigg\|_{L^{p_2}(w_2)}.
	\]
\end{Corollary}

We note that in the previous result we must have $p>\frac23$ (and also $s>\frac23$). Indeed, the fact that $\vec{r}\prec\vec{p}$ and \eqref{cond-adm:corol} give
\[
\frac1p
=
\frac1{p_1}+\frac1{p_2}
<
\frac1{r_1}+\frac1{r_2}
\le
\frac1{\min\{r_1,2\}}+\frac1{\min\{r_2,2\}}
<
2-\frac1{\min\{r_3,2\}}
\le
\frac32.
\]

Corollary \ref{corol:BH-1st} can be reformulated in the following  equivalent form (details are left to the interested reader):  	given $\vec{p}=(p_1,p_2)$ with $1<p_1, p_2<\infty$ and setting $\frac1p:=\frac1{p_1}+\frac1{p_2}$,  if there exist 
$0\le \gamma_1,\gamma_2,\gamma_3<1$ with $\gamma_1+\gamma_2+\gamma_3=1$ such that
\begin{equation}\label{eq:new-cond:1a}
\frac{1}{p_1}<\frac{1+\gamma_1}{2},
\qquad
\frac{1}{p_2}<\frac{1+\gamma_2}{2},
\qquad
\frac{1}{p}>\frac{1-\gamma_3}{2},
\end{equation}
then $BH: L^{p_1}(w_1)\times L^{p_2}(w_2)\to L^p(w)$ for every $\vec{w}=(w_1,w_2)\in A_{\vec{p},\vec{r}}$ with $\vec{r}=(\frac2{1+\gamma_1},\frac2{1+\gamma_2},\frac2{1+\gamma_3})$ and where  $w=w_1^{\frac{p}{p_1}}w_2^{\frac{p}{p_2}}$. In the vector-valued case, if we further take 
$\vec{s}=(s_1,s_2)$  and let $\frac1s:=\frac1{s_1}+\frac1{s_2}$, the assumption
\begin{equation}\label{eq:new-cond-vv:1}
\max\left\{\frac{1}{s_1},\frac{1}{p_1}\right\}<\frac{1+\gamma_1}{2},
\ \
\max\left\{\frac{1}{s_2},\frac{1}{p_2}\right\}<\frac{1+\gamma_2}{2},
\ \
\min\left\{\frac{1}{s},\frac{1}{p}\right\}>\frac{1-\gamma_3}{2},
\end{equation}
yields as well $BH: L^{p_1}_{\ell^{s_1}}(w_1)\times L^{p_2}_{\ell^{s_2}}(w_2)\to L_{\ell^s}^p(w)$. These should be compared with \cite[Section 5]{CM-extrapol} where some extra restrictions on the exponent are present due to the fact that the extrapolation there is done with product weights. Moreover, we extend and reprove some results obtain by Benea and Muscalu in \cite[Section 6.3]{BM3} (see also \cite[Section 1.4]{BM4}). For instance, the previous estimates for $p<1$ reprove \cite[Proposition 19]{BM3}, and also extend \cite[Corollary 21]{BM3},
which gives vector-valued weighted norm inequalities with product weights rather than with the more general class $A_{\vec{p},\vec{r}}$.  Note that in contrast with the helicoidal method \cite{BM1, BM2, BM3, BM4}, we do not consider the cases where some $p_i$'s or $s_i$'s are infinity. These will be treated in the forthcoming paper \cite{LMO-end}.

One can put easy examples of weights for which the previous estimates hold. For instance,  given $\vec{p}=(p_1,p_2)$ with $1<p_1, p_2<\infty$ such that $\frac1p:=\frac1{p_1}+\frac1{p_2}<\frac32$,  
\begin{equation}
BH : L^{p_1}(|x|^{-a}) \times L^{p_2}(|x|^{-a}) 
\longrightarrow L^p(|x|^{-a}),
\label{eq:BH-power}
\end{equation}
if $a=0$ or if 
\begin{multline*}\label{eq:values-a}
1-\min\left\{\max\left\{1,\frac{p_1}{2}\right\},\max\left\{1,\frac{p_2}{2}\right\}\right\}
< 
a
\\
<1-p\left(\max\left\{0, \frac1{p_1}-\frac12\right\}+\max\left\{0, \frac1{p_2}-\frac12\right\}\right).
\end{multline*}
As a result, \eqref{eq:BH-power} holds for all $0\le a<\frac12$. This extends \cite[Corollary 1.23]{CM-extrapol}. To prove this, one easily sees that $\vec{w}=(|x|^{-a},|x|^{-a})\in A_{\vec{q},\vec{r}}$ with $\vec{r}\prec{q}$ if and only if
\[
1-\min\left\{\frac{q_1}{r_1}, \frac{q_2}{r_2}\right\} 
< 
a
<1-\frac{q}{r_3'}.
\]
Using this and choosing (roughly) $\gamma_1=\max\{0,\frac2{p_1}-1\}$, $\gamma_2=\max\{0,\frac2{p_1}-1\}$, $\gamma_3=1-\gamma_1-\gamma_2$ one can obtain the desired estimate for $BH$. Analogously, given  $\vec{p}=(p_1,p_2)$, $\vec{s}=(s_1,s_2)$ with $1<p_1, p_2, s_1,s_2<\infty$ such that $\frac1p:=\frac1{p_1}+\frac1{p_2}<\frac32$ and $\frac1s:=\frac1{s_1}+\frac1{s_2}<\frac32$
we have that
\begin{multline} \label{eqn:bht-vector1:pwerw-wts}
\bigg\| \bigg(\sum_k |BH(f_k,g_k)|^s\bigg)^{\frac{1}{s}}
\bigg\|_{L^p(|x|^{-a})} \\
\leq 
C\bigg\| \bigg(\sum_k |f_k|^{s_1}\bigg)^{\frac{1}{s_1}}
\bigg\|_{L^{p_1}(|x|^{-a} )} 
\bigg\| \bigg(\sum_k |g_k|^{s_2}\bigg)^{\frac{1}{s_2}}
\bigg\|_{L^{p_2}(|x|^{-a})}. 
\end{multline}
holds if 
\begin{multline*}
1-\min \left\{ 
\max\left\{1,\frac{p_1}2, \frac{p_1}{s_1}\right\},
\max\left\{1,\frac{p_2}2, \frac{p_2}{s_2}\right\}
\right\}
<
a
\\
<1-p\left(\max\left\{0, \frac1{p_1}-\frac12,\frac1{s_1}-\frac12\right\}+\max\left\{0, \frac1{p_2}-\frac12,\frac1{s_2}-\frac12\right\}\right).
\end{multline*}
Notice that this interval could be empty since even in the unweighted situation there are some natural restrictions for the vector valued inequalities to hold.

\begin{Remark} \label{remark:iteration:BHT}
	In \cite{BM1, BM2, BM3, BM4} the authors also prove iterated vector-valued inequalities such as $BH:L^{p_1}_{\ell^{s_1}_{\ell^{t_1}}}\times L^{p_2}_{\ell^{s_2}_{\ell^{t_2}}}\to L^{p}_{\ell^{s}_{\ell^{t}}}$, again with restrictions on the possible values of the $p_i$ depending
	on the $s_i$ and $t_i$.  Our method gives as a corollary these inequalities and their corresponding weighted versions (when all the exponents are finite) by extrapolation (see Remark \ref{remark:iteration}). The precise statements are left to the interested reader.
\end{Remark}

\subsection{Commutators with BMO functions}\label{section:comm}

Our extrapolation result also gives estimates for commutators with BMO functions. Recently, \cite[Theorem 4.13]{BMMST}  showed that if a multilinear operator $T$ maps continuously $L^{p_1}(w_1)\times\dots\times L^{p_m}(w_m)$ into $L^p(w)$ for some $1<p_1,\dots p_m<\infty$ and $1<p<\infty$ with  $\frac1p:=\frac1{p_1}+\dots+\frac1{p_m}$ and for all $w=(w_1,\dots,w_m)\in A_{\vec{p}}$ where $w:=\prod_{i=1}^mw_i^{\frac{p}{p_i}}$, then the multilinear commutators with BMO functions satisfy the very same inequalities. Our extrapolation result applied to the hypotheses immediately yields that we can remove the restriction $p>1$ as all the weighted estimates are equivalent to a single one. Moreover, if we extrapolate from the conclusion we can also extend the weighted estimates to the quasi-Banach range. Here it is important to emphasize that the method extensively developed in \cite{BMMST} elaborates on the commonly used Cauchy integral trick which in turn uses Minkowski's inequality, hence it requires to work in the Banach range. Nonetheless, our extrapolation result gives a posteriori that such restriction can be removed. 

Let us thus begin by defining the main objects that we will be dealing with in this setting. We recall here the definition of the John-Nirenberg space of functions of bounded mean oscillation. We say that a locally integrable function $b\in {\rm BMO}$ if
\[
\|b\|_{{\rm BMO}} := \sup_Q \dashint_Q |b - b_Q| \,dx < + \infty,
\]
where the supremum is taken over the collection of all cubes $Q\subset\mathbb{R}^n$ and where $b_Q=\dashint_Q b dx$. 

Let $T$ denote an $m$-linear operator from $X_1 \times \cdots \times X_m$ into $Y$, where $X_j, 1\le j\le m$, and $Y$ are some normed spaces. In our statements the $X_j$ and $Y$ will be appropriate weighted Lebesgue spaces. For $(f_1, f_2,\dots, f_m)\in X_1\times X_2\times\cdots \times X_m$ and for a measurable vector $\textbf{b}=(b_1, b_2, \dots, b_m)$, and $1\leq j\leq m$, we define, whenever it makes sense, the (first order) commutators
\[
[T, \textbf{b}]_{e_j}(f_1, f_2,\dots, f_m)=b_jT(f_1,\dots, f_j,\dots f_m)-T(f_1,\dots, b_jf_j,\dots f_m);
\]
we denoted by $e_j$ the basis element taking the value $1$ at component $j$ and $0$ in every other component, therefore expressing the fact that the commutator acts as a linear one in the $j$-th variable and leaving the rest of the entries of $(f_1, f_2,\dots, f_m)$ untouched. Then, if $k\in\mathbb N$, we define
\[
[T, \textbf{b}]_{k e_j} = [ \cdots [ [ T, \textbf{b}]_{e_j}, \textbf{b}]_{e_j} \cdots, \textbf{b}]_{e_j},
\]
where the commutator is performed $k$ times. Finally, if
$\alpha = (\alpha_1, \alpha_2,\dots, \alpha_m) \in \mathbb (\mathbb N\cup \{0\})^m$ is a multi-index, we define
\begin{align}
[T, \textbf{b}]_\alpha = [ \cdots [ [T, \textbf{b}]_{\alpha_1 e_1}, \textbf{b}]_{\alpha_2 e_2} \cdots, \textbf{b}]_{\alpha_m e_m}.
\end{align}
Informally, if the multilinear operator $T$ has a kernel representation of the form
\[
T(f_1, f_2,\dots, f_m)(x)=\int_{\mathbb R^{nm}}K(x, y_1,\dots, y_m)f_1(y_1)\cdots f_m(y_m)dy_1\dots dy_m,
\]
then $[T, \textbf{b}]_\alpha(f_1, f_2,\dots, f_m)(x)$ can be expressed in a similar way, with kernel
$$\prod_{j=1}^m (b_j(x)-b_j(y_j))^{\alpha_j}K(x, y_1,\dots, y_m).$$

Next we present our promised application for commutators in the context of the classes $A_{\vec{p},\vec{r}}$. Here it is important to emphasize that our only assumption on $T$, besides the initial weighted norm inequalities, is that $T$ is multilinear.

\begin{Theorem}\label{thm:multilinear:comm}
	Let $T$ be an $m$-linear operator and let $\vec{r}=(r_1,\dots,r_{m+1})$, with $1\le r_1,\dots,r_{m+1}<\infty$. Assume that there exists $\vec p=(p_1,\dots, p_m)$, with $1\le p_1,\dots, p_m<\infty$ and $\vec{r}\preceq\vec{p}$, such that for all $\vec{w} = (w_1, \dots, w_m)\in A_{\vec{p},\vec{r}}$, we have
	\begin{equation}
	\label{vector-weight}
	\|T(f_1, f_2,\dots, f_m)\|_{L^p (w)} \lesssim \prod_{i=1}^m \|f_i\|_{L^{p_i}\left(w_i\right)},
	\end{equation}
	where $\frac1p:=\frac1{p_1}+\dots+\frac1{p_m}$ and $w:=\prod_{i=1}^mw_i^{\frac{p}{p_i}}$.
	
	Then, 	for all exponents $\vec q=(q_1,\dots,q_m)$, with $\vec{r}\prec\vec{q}$, for all weights $\vec v=(v_1,\dots, v_m) \in A_{\vec q,\vec{r}}$, for all $\textbf{b} = (b_1, \dots, b_m) \in  {\rm BMO}^m$, and for each multi-index $\alpha$, we have
	\begin{equation}
	\label{multi-commutator-II}
	\|[T, \textbf{b}]_\alpha (f_1, f_2,\dots, f_m)\|_{L^q (v)}
	\lesssim
	\prod_{i=1}^m \|b_j\|^{\alpha_i}_{{\rm BMO}} \|f_i\|_{L^{q_i}\left(v_i\right)},
	\end{equation}
		where $\frac1q:=\frac1{q_1}+\dots+\frac1{q_m}$ and $v:=\prod_{i=1}^m v_i^{\frac{q}{q_i}}$. Moreover if  $\vec s=(s_1,\dots,s_m)$, with $\vec{r}\prec\vec{s}$,
		then
\begin{equation}\label{multi-comm:VV}
\bigg\|\Big(\sum_j |[T, \textbf{b}]_\alpha (f_1^j, f_2^j,\dots, f_m^j)|^s\Big)^\frac1s\bigg\|_{L^{q}(v)}
\lesssim
\prod_{i=1}^m \|b_j\|^{\alpha_i}_{{\rm BMO}} \bigg\|\Big(\sum_j |f_i^j|^{s_i}\Big)^\frac1{s_i}\bigg\|_{L^{q_i}(v_i)},
\end{equation}
$\frac1s:=\frac1{s_1}+\dots+\frac1{s_m}$. 		
\end{Theorem}

The proof is postponed until Section \ref{section:proof-comm}. From this  and \cite{LOPTT} we can trivially obtained the following result which extends \cite{BMMST}:

\begin{Corollary}
Let $T$ be  an $m$-linear Calderón-Zygmund operator. Then, \eqref{multi-commutator-II} and \eqref{multi-comm:VV} hold
for all exponents $\vec q=(q_1,\dots,q_m)$,  $\vec{s}=(s_1,\dots, s_m)$,  with $1<q_i,s_i<\infty$, $1\le i\le \infty$, for all weights $\vec v \in A_{\vec q }$, for all $\textbf{b} = (b_1, \dots, b_m) \in  {\rm BMO}^m$, and for each multi-index $\alpha$. 
\end{Corollary}

From Corollary \ref{corol:sparese-general} and Theorem \ref{thm:multilinear:comm} one can easily obtain the following:

\begin{Corollary}\label{corol:sparese-general:BMO}
Given $\vec{r}=(r_1,\dots,r_{m+1})$, with  $r_i\ge 1$ for $1\le i\le m+1$ and $\sum_{i=1}^{m+1}\frac 1{r_i}>1$, and an $m$-linear operator 	$T$ satisfying 	\eqref{sparse-domination}, it follows that
\eqref{multi-commutator-II} and \eqref{multi-comm:VV} hold for all exponents $\vec q=(q_1,\dots,q_m)$, $\vec s=(s_1,\dots,s_m)$, with $\vec{r}\prec\vec{q}$ and $\vec{r}\prec\vec{s}$, for all weights $\vec v=(v_1,\dots, v_m)  \in A_{\vec q,\vec{r}}$, for all $\textbf{b} = (b_1, \dots, b_m) \in  {\rm BMO}^m$, and for each multi-index $\alpha$. 
\end{Corollary}

Our last application of Theorem \ref{thm:multilinear:comm}, with the help of Corollary \ref{corol:BH-1st}, solves a problem about the boundedness of the commutators of the bilinear
Hilbert transform with functions in BMO which as far as we know can not be obtained using other methods.

\begin{Corollary}\label{corol:BH-1st:BMO}
Assume that $\vec{r}=(r_1,r_2,r_3)$, $1<r_1,r_2,r_3<\infty$, verifies \eqref{cond-adm:corol}. 
For all exponents $\vec{p}=(p_1,p_2)$, $\vec{s}=(s_1,s_2)$  with $\vec{r}\prec\vec{p}$ and $\vec{r}\prec\vec{s}$ where $\frac1p=\frac1{p_1}+\frac1{p_2}$ and $\frac1s=\frac1{s_1}+\frac1{s_2}$, for all weights $\vec{w}= (w_1, w_2)\in A_{\vec{p},\vec{r}}$,  for all $\textbf{b} = (b_1, b_2) \in  {\rm BMO}^2$, and for each multi-index $\alpha=(\alpha_1,\alpha_2)$ it follows that
\[
\|[BH, \textbf{b}]_\alpha (f,g)\|_{L^p (w)}
\lesssim
\|b_1\|^{\alpha_1}_{{\rm BMO}} \|b_2\|^{\alpha_2}_{{\rm BMO}} \|f\|_{L^{p_1}\left(w_1\right)}\|g\|_{L^{p_2}\left(w_2\right)},
\]
and
\begin{multline*}
\bigg\|\Big(\sum_j |[BH, \textbf{b}]_\alpha (f_j, g_j)|^s\Big)^\frac1s\bigg\|_{L^{p}(w)}
\\
\lesssim
\|b_1\|^{\alpha_1}_{{\rm BMO}} \|b_2\|^{\alpha_2}_{{\rm BMO}} \bigg\|\Big(\sum_j |f_j|^{s_1}\Big)^\frac1{s_1}\bigg\|_{L^{p_1}(w_1)}
\bigg\|\Big(\sum_j |g_j|^{s_2}\Big)^\frac1{s_2}\bigg\|_{L^{p_2}(w_2)},
\end{multline*}
where $w:=w_1^{\frac{p}{p_1}}w_1^{\frac{p}{p_2}}$. 
\end{Corollary}

 Note that all the previous estimates admit iterated weighted vector-valued extensions along the lines pointed out in Remark \ref{remark:iteration}. We leave the details to the interested reader.

\section{Auxiliary results}\label{section:aux}

We first introduce some notation. Given a  cube $Q$, its side-length will be denoted by $\ell(Q)$ and for  any $\lambda>0$ we let  $\lambda Q$ be
the cube concentric with $Q$ whose side-length is $\lambda \ell
(Q)$. Let $\mu$ be a doubling measure on $\re^n$, that is, $\mu$ is a non-negative Borel regular measure such that $\mu(2Q)\le C_\mu \mu(Q)<\infty$ for every cube $Q\subset \re^n$. Given a Borel set $E\subset\re^n$ with $0<\mu(E)<\infty$ we use the notation
\[
\dashint_E f d\mu=\frac1{\mu(E)}\int_E fd\mu.
\]

Next we give the basic properties of weights that we will
need below.  For proofs and further information,
see~\cite{duoandikoetxea01,grafakos08b}.   By a weight we mean a
measurable  function $v$ such that
$0<v<\infty$ $\mu$-a.e.   For $1<p<\infty$, we say that $v\in
A_p(\mu)$ if 
\[ 
[v]_{A_p(\mu)} = \sup_Q \dashint_Q v\,d\mu 
\left(\dashint_Q v^{1-p'}\,d\mu\right)^{p-1} < \infty, \]
where the supremum is taken over all cubes $Q\subset \re^n$. The quantity $[v]_{A_p(\mu)}$ is called
the $A_p(\mu)$ constant of $v$.  Notice that it follows at once from this
definition that if $v\in A_p(\mu)$, then $v^{1-p'}\in A_{p'}(\mu)$.  When $p=1$
we say that $v\in A_1(\mu)$ if
\[  [v]_{A_1(\mu)} = \sup_Q \dashint_Q v\,d\mu \esssup_{Q} v^{-1}<
\infty, \]
where the essential supremum is taken with respect to the underlying doubling measure $\mu$. 
The $A_p(\mu)$ classes are properly nested:  for $1<p<q$, $A_1(\mu)\subsetneq
A_p(\mu) \subsetneq A_q(\mu)$.  
We denote the union of all the $A_p(\mu)$ classes, $1\leq p<\infty$, by
$A_\infty(\mu)$.  

Given $1\le p<\infty$ and $0<r<\infty$ we say that $v\in A_{p,r}(\mu)$ if 
\[ 
[v]_{A_{p,r}(\mu)}
=
\sup_Q  \dashint_Q v^{r}\,d\mu
\left( \dashint_Q v^{-p'}\,dx\right)^{\frac{r}{p'}}<\infty, \]
when $p>1$ and  
\[ 
[v]_{A_{p,r}(\mu)}
=
\sup_Q  \dashint_Q v^{r}\,d\mu
\esssup_{Q} v^{-r}<\infty
\]
when $p=1$. Notice that clearly $v\in A_{p,r}(\mu)$ if and only if $v^r\in A_{1+\frac{r}{p'}}(\mu)$ with $[v]_{A_{p,r}(\mu)}=[v^r]_{A_{1+\frac{r}{p'}}(\mu)}$.

When $\mu$ is the Lebesgue measure we will simply write $A_p$, $A_{p,r}$, \dots. It is well-known that if $w\in A_\infty$ then $dw=w(x)dx$ is a doubling measure. Besides, since $0<w<\infty$ a.e. then the Lebesgue measure and  $w$ have the same null measure sets hence the essential suprema and infima with respect to the Lebesgue measure and $w$ agree.

To prove our main result we need some off-diagonal extrapolation theorem proved by Duoandikoetxea in \cite{Duo} for the Lebesgue measure but whose proof readily extends to any underlying doubling measure. 
\begin{Theorem}[{\cite[Theorem 5.1]{Duo}}]\label{theor:duo}
Let $\mu$ be a doubling measure on $\re^n$, $n\ge 1$, and let $\mathcal{F}$ be a family of pairs $(F,G)$ of non-negative Borel functions. 
Let $1\le p_0<\infty$ and $0<q_0,r_0<\infty$ and assume that for all $w\in A_{p_0, r_0}(\mu)$ and for all $(F,G)\in\mathcal{F}$ we have the inequality
\[
\Big( \int_{\mathbb R^n} F^{q_0}w^{q_0} \d\mu\Big)^{\frac 1{q_0}}
\le
N([w]_{A_{p_0,r_0}})
\Big( \int_{\mathbb R^n} G^{p_0}w^{p_0} \d\mu\Big)^{\frac1{p_0}},
\]
where $N$ is an increasing function. Then there exists another increasing function $\widetilde{N}$ such that for all $1< p<\infty$ and $0<q,r<\infty$ verifying
\[
\frac 1q-\frac 1{q_0}=\frac 1r-\frac 1{r_0}=\frac 1p-\frac 1{p_0},
\]
for all $w\in A_{p,r}(\mu)$, and for all $(F,G)\in\mathcal{F}$ we have
\[
\Big( \int_{\mathbb R^n} F^{q}w^{q} \d\mu\Big)^{\frac 1{q}}
\le
\widetilde{N}([w]_{A_{p,r}})
\Big( \int_{\mathbb R^n} G^{p}w^{p} \d\mu\Big)^{\frac 1{p}}.
	\]
\end{Theorem}

\medskip

In preparation for proving our main result we need some notation.
Given  $\vec p=(p_1,\dots, p_m)$ with $1\le p_1,\dots,p_m<\infty$ and $\vec{r}=(r_1,\dots,r_{m+1})$
with $1\le r_1,\dots,r_{m+1}<\infty$ so that $\vec{r}\preceq\vec{p}$ we set 
\[
\frac1r:=\sum_{i=1}^{m+1} \frac1{r_i}, 
\qquad
\frac1{p_{m+1}}:=1-\frac1p,
\qquad\mbox{and}\qquad
\frac1{\delta_i}=\frac1{r_i}-\frac1{p_i},
\quad i=1,\dots, m+1.
\]
Notice that as observed above we have that $0<r<1$ and formally $\frac1{p_{m+1}}=\frac1{p'}$ which  could be negative or zero if $p\le 1$. Note that in this way
\[
\sum_{i=1}^{m+1}\frac1{p_i}=1
\qquad\mbox{and}\qquad
\sum_{i=1}^{m+1}\frac1{\delta_i}=\frac1{r}-1=\frac{1-r}{r}.
\]
Also, $\vec{r}\preceq\vec{p}$ means that $r_i\le p_i$, hence $\delta_i^{-1}\ge 0$, for every $1\le i\le m$ and
$r_{m+1}<p_{m+1}$, that is, $\delta_{m+1}^{-1}>0$. On the other hand, $\vec{r}\prec\vec{p}$ means that $r_i< p_i$ or $\delta_i^{-1}>0$ for every $1\le i\le m+1$.
Notice that with this notation $w\in A_{\vec{p},\vec{r}}$  can be written as 
\[
[\vec{w}]_{A_{\vec p, \vec r}}
=
\sup_Q\Big(\dashint_Q w^{\frac {\delta_{m+1}}{p}}\d x\Big)^{\frac 1{\delta_{m+1}}}
\prod_{i=1}^m \Big(\dashint_Q w_i^{-\frac{\delta_i}{p_i}}\d x\Big)^{\frac1{\delta_i}}
<\infty,
\]
and, when $p_i=r_i$ (i.e., $\delta_i^{-1}=0$), we need to replace the corresponding term with $\esssup_Q w_i^{-\frac 1{p_i}}$.

\bigskip

The following lemma gives a new characterization of the weighted class $A_{\vec p, \vec r}$, which is of independent interest. Moreover, as we will see later, it will allow us 
to prove our multivariable extrapolation result from a one-variable extrapolation result with some underlying measure depending on the weights.

\begin{Lemma}\label{lemma:main}
Let  $\vec p=(p_1,\dots, p_m)$ with $1\le p_1,\dots,p_m<\infty$ and $\vec{r}=(r_1,\dots,r_{m+1})$
with $1\le r_1,\dots,r_{m+1}<\infty$ be such that $\vec{r}\preceq\vec{p}$. Using the previous notation we set
\begin{equation}\label{def:varrho}
\frac1{\varrho}:=\frac 1{r_m}-\frac 1{r_{m+1}'}+\sum_{i=1}^{m-1}\frac1{p_i}
=
\frac1{\delta_m}+\frac1{\delta_{m+1}}
>0,
\end{equation}
 and for every $1\le i\le m-1$
\begin{equation}\label{def:theta}
\frac1{\theta_i}
:= 
\frac{1-r}r-\frac1{\delta_i} 
=
\left(\sum_{j=1}^{m+1} \frac1{\delta_j}\right)-\frac1{\delta_i}
>0.
\end{equation}
Then the following hold:
\begin{list}{$(\theenumi)$}{\usecounter{enumi}\leftmargin=.8cm
	\labelwidth=.8cm\itemsep=0.2cm\topsep=.1cm
	\renewcommand{\theenumi}{\roman{enumi}}}

\item Given  $\vec{w}=(w_1,\dots,w_m)\in A_{\vec p, \vec r}$, write $w:=\prod_{i=1}^mw_i^{\frac{p}{p_i}}$ and
set
\begin{equation}\label{formula-1}
\widehat{w}:=\Big(\prod_{i=1}^{m-1}w_i^{\frac1{p_i}}\Big)^{\varrho}
\qquad
\mbox{and}
\qquad
W
:=
w^{\frac{r_m}{p}} \widehat{w}^{- \frac{r_m}{\delta_{m+1}}}
=
w_m^{\frac{r_m}{p_m}}\widehat{w}^{\frac{r_m}{\delta_m}}
\end{equation}
Then, 
\begin{list}{$(\theenumi.\theenumii)$}{\usecounter{enumii}\leftmargin=.4cm
		\labelwidth=.8cm\itemsep=0.2cm\topsep=.1cm
		\renewcommand{\theenumii}{\arabic{enumii}}}

\item $w_i^{\frac{\theta_i }{p_i}}\in A_{\frac{1-r}{r}\theta_i }$ with $\Big[w_i^{\frac {\theta_i }{p_i}}\Big]_{ A_{\frac{1-r}{r}\theta_i }} \le [\vec w]_{A_{\vec p, \vec r}}^{\theta_i }$, for every $1\le i\le m-1$.

\item $\widehat{w}\in A_{\frac{1-r}{r}\varrho}$ with 
$[\widehat{w}]_{A_{\frac{1-r}{r}\varrho}}\le [\vec w]_{A_{\vec p, \vec r}}^{\varrho}$.

\item $W\in A_{\frac{p_m}{r_m},\frac{\delta_{m+1}}{r_m}}(\widehat{w})$ 
with $[W]_{A_{\frac{p_m}{r_m},\frac{\delta_{m+1}}{r_m}} (\widehat{w})}
\le
[\vec w]_{A_{\vec p, \vec r}}^{\delta_{m+1}}$.
\end{list}

\item Given $w_i^{\frac {\theta_i}{p_i}}\in A_{\frac{1-r}{r}\theta_i}$,  $1\le i\le m-1$, 
such  that 
\begin{equation}\label{formula2-1}
\widehat{w}=
\Big(\prod_{i=1}^{m-1}w_i^{\frac1{p_i}}\Big)^{\varrho}
\in A_{\frac{1-r}{r}\varrho}
\end{equation}
and  
$W\in A_{\frac{p_m}{r_m},\frac{\delta_{m+1}}{r_m}} (\widehat{w})$, let us set 
\begin{equation}\label{formula-2}
w_m
:=
W^{\frac{p_m}{r_m}} \widehat{w}^{\frac{p_m}{\delta_m}}.
\end{equation}
Then $\vec{w}=(w_1,\dots, w_m)\in A_{\vec p, \vec r}$ and, moreover,
\[
[\vec w]_{A_{\vec p, \vec r}}
\le 
[W]_{A_{\frac{p_m}{r_m},\frac{\delta_{m+1}}{r_m}} (\widehat{w})}^{\frac1{\delta_{m+1}}}
[\widehat{w}]_{A_{\frac{1-r}{r}\varrho}}^{\frac1 {\varrho}}
\prod_{i=1}^{m-1}\Big[w_i^{\frac {\theta_i}{ p_i}}\Big]_{A_{\frac{1-r}{r}\theta_i}}^{\frac1{\theta_i}}.
\]

\item For any measurable function $f\ge0 $ and in the context of $(i)$ or $(ii)$ there hold
\begin{equation}\label{LHS-rew}
\|f\|_{L^p(w)}
=
\Big\|\Big(f\widehat{w}^{-\frac1{r_{m+1}'}}\Big)^{r_m}\Big\|_{L^\frac{p}{r_m}(W^{\frac{p}{r_m}}\d\widehat{w})}^{\frac1{r_m}}
\end{equation}
and
\begin{equation}\label{RHS-rew}
\|f\|_{L^{p_m}(w_m)}
=
\Big\|\Big(f\widehat{w}^{-\frac1{r_m}}\Big)^{r_m}\Big\|_{L^\frac{p_m}{r_m}(W^{\frac{p_m}{r_m}}\d\widehat{w})}^{\frac1{r_m}}.
\end{equation}
\end{list}
\end{Lemma}

\begin{proof}
	
We start observing that equalities \eqref{formula-1} in $(i)$, or \eqref{formula2-1} and \eqref{formula-2} in case $(ii)$, easily yield

\begin{equation}\label{qrfera}
\dashint_Q W^{\frac{\delta_{m+1}}{r_m}} \d \widehat w
=
\Big(\dashint_Q \widehat{w}\d x\Big)^{-1}\Big(\dashint_Q w^{\frac{\delta_{m+1}}{p}}
\d x\Big)
\end{equation}
and whenever $\delta_m^{-1}\neq 0$ (i.e., $r_m<p_m$)
\begin{align*}
\dashint_Q  W^{-(\frac{p_m}{r_m})'}\d \widehat w
&=
\dashint_Q  w_m^{-\frac{r_m}{p_m}(\frac {p_m}{r_m})'}\widehat{w}^{\frac{r_m}{\delta_m}(\frac {p_m}{r_m})'}
\d
\widehat w
=
\Big(\dashint_Q \widehat{w}\d x\Big)^{-1}
\Big(
\dashint_Q  w_m^{-\frac{\delta_m}{p_m}} \d x
\Big).
\end{align*}%
These equalities yield if $\delta_m^{-1}\neq0$ 
\begin{multline}\label{aaa:1}
\Big(
\dashint_Q W^{\frac{\delta_{m+1}}{r_m}}   \d \widehat w
\Big)
\Big(
\dashint_Q  W^{-(\frac{p_m}{r_m})'}\d \widehat w
\Big)^{\frac{\frac{\delta_{m+1}}{r_m}}{(\frac{p_m}{r_m})'}}
\\
=
\Big(\dashint_Q \widehat{w} \d x\Big)^{-1-\frac{\delta_{m+1}}{\delta_m}}
\Big(\dashint_Q w^{\frac{\delta_{m+1}}{p}} \d x\Big)
\Big(
\dashint_Q  w_m^{-\frac{\delta_m}{p_m}} \d x
\Big)^{\frac{\delta_{m+1}}{\delta_m}}
\end{multline}
Thus, if $\delta_m^{-1}\neq0$  and $\sum_{i=1}^{m-1}\frac1{\delta_i}>0$  then
\begin{align}\label{aaa:2}
&\Big(
\dashint_Q W^{\frac{\delta_{m+1}}{r_m}}   \d \widehat w
\Big)
\Big(
\dashint_Q  W^{-(\frac{p_m}{r_m})'}\d \widehat w
\Big)^{\frac{\frac{\delta_{m+1}}{r_m}}{(\frac{p_m}{r_m})'}}
\\
&\quad =
\bigg[
\Big(\dashint_Q \widehat{w}^{1-\big(\frac{1-r}{r}\varrho\big)'}\d x\Big)^{\frac1{\varrho}(\frac{1-r}{r}\varrho-1)}
\Big( \dashint_Q w^{\frac{\delta_{m+1}}{p}}  \d x \Big)^{\frac 1{\delta_{m+1}}}
\Big( \dashint_Q   w_m^{-\frac{\delta_m}{p_m}}\d x \Big)^{\frac1{\delta_m}}
\\
&\hskip3.5cm\times
\Big(
\dashint_Q \widehat{w}\d x
\Big)^{-\frac1{\varrho}}
\Big(\dashint_Q \widehat{w}^{1-\big(\frac{1-r}{r}\varrho\big)'}\d x\Big)^{-\frac1{\varrho}(\frac{1-r}{r}\varrho-1)}
\bigg]^{\delta_{m+1}}.
\end{align}
On the other hand, if  $\sum_{i=1}^{m-1}\frac1{\delta_i}>0$ and $\delta_m^{-1}=0$
\begin{align}\label{aaa:3}
&\Big(
\dashint_Q W^{\frac{\delta_{m+1}}{r_m}}   \d \widehat w
\Big)
=
\Big(\dashint_Q \widehat{w}\d x\Big)^{-1}\Big(\dashint_Q w^{\frac{\delta_{m+1}}{p}}
\d x\Big)
\\
&\ =
\bigg[
\Big(\dashint_Q \widehat{w}^{1-\big(\frac{1-r}{r}\varrho\big)'}\d x\Big)^{\frac1{\varrho}(\frac{1-r}{r}\varrho-1)}
\Big( \dashint_Q w^{\frac{\delta_{m+1}}{p}}  \d x \Big)^{\frac 1{\delta_{m+1}}}
\\
&\hskip3.5cm\times
\Big(
\dashint_Q \widehat{w}\d x
\Big)^{-\frac1{\varrho}}
\Big(\dashint_Q \widehat{w}^{1-\big(\frac{1-r}{r}\varrho\big)'}\d x\Big)^{-\frac1{\varrho}(\frac{1-r}{r}\varrho-1)}
\bigg]^{\delta_{m+1}}
\end{align}
 since $\varrho=\delta_{m+1}$.

\medskip

We proceed to establish $(i)$. Assume  that $\vec{w}\in A_{\vec p, \vec r}$. To see $(i.1)$ we fix $1\le i\le m-1$ and set 
$\mathcal{I}=\{j: 1\le j\le m, \delta_j^{-1}\neq 0\}$ and $\mathcal{I}'=\{1,\dots,m\}\setminus\mathcal{I}$. Set 
$\frac1{\eta_i}:=\frac{\theta_i}{\delta_{m+1}}$ and $\frac1{\eta_j}:=\frac{\theta_i}{\delta_{j}}$ for $j\in\mathcal{I}$ with $j\neq i$ so that 
$\frac1{\eta_i}+\sum_{i\neq j\in\mathcal{I}}\frac1{\eta_j}=1$. Thus, Hölder's inequality easily gives
\begin{multline*}
\dashint_Q w_i^{\frac{\theta_i}{p_i}} dx
= 
\dashint_Q 
\Big( w^{\frac{\theta_i}{p}}\prod_{\substack{1\le j \le m\\j\neq i} }w_j^{-\frac { \theta_i }{p_j}}\Big)dx
\\
\le 
\Big( \dashint_Q w^{\frac{\theta_i}{p}\eta_i}dx\Big)^{\frac{1}{\eta_i}} 
\Big(\prod_{i\neq j\in\mathcal{I}}  \Big( \dashint_Q   w_j^{-\frac {\theta_i }{p_j}\eta_j}  \Big)^{\frac{1}{\eta_j}}\Big)
\Big(\prod_{i\neq j \in\mathcal{I}'} \esssup_{Q} w_j^{-\frac { \theta_i }{p_j}} \Big)
\\
= 
\Big( \dashint_Q w^{\frac{\delta_{m+1}}{p}}dx\Big)^{\frac{\theta_i}{\delta_{m+1}}} 
\Big(\prod_{i\neq j\in\mathcal{I}} \Big( \dashint_Q   w_j^{-\frac {\delta_j}{p_j}}  \Big)^{\frac{\theta_i}{\delta_j}}\Big)
\Big(\prod_{i\neq j \in\mathcal{I}'} \esssup_{Q} w_j^{-\frac { 1}{p_j}} \Big)^{\theta_i}.
\end{multline*} 
When $p_i=r_i$ then $\theta_i=\frac{r}{1-r}$ and this inequality readily gives 
$w_i^{\frac {\theta_i}{p_i}}\in A_{1}$ with $\Big[w_i^{\frac {\theta_i }{p_i}}\Big]_{ A_1} \le [\vec w]_{A_{\vec p, \vec r}}^{\theta_i }$. When $p_i>r_i$ we just need to observe that $\frac{\theta_i}{p_i}(1-(\frac{1-r}{r}\theta_i)')=- \frac{\delta_i}{p_i}$ and again we eventually see that $w_i^{\frac {\theta_i}{p_i}}\in A_{\frac{1-r}{r}\theta_i}$ with $\Big[w_i^{\frac {\theta_i }{p_i}}\Big]_{ A_{\frac{1-r}{r}\theta_i }} \le [\vec w]_{A_{\vec p, \vec r}}^{\theta_i }$.

\medskip

To obtain $(i.2)$ we need to consider three cases. 

\medskip

\noindent\textbf{Case 1:} $\sum_{i=1}^{m}\frac1{\delta_i}=0$, that is,  $p_j=r_j$ for $1\le j\le m$.

In this case $\frac1{\varrho}=\frac1{\delta_{m+1}}=\frac{1-r}{r}$ and we can easily see that $\widehat{w}\in A_1$ with the right bound:
\begin{multline*}
\dashint_Q \widehat{w} \d x
=
\dashint_Q w^{\frac{\delta_{m+1}}{p}}w_m^{-\frac{\delta_{m+1}}{p_m}} \d x
\le
\dashint_Q w^{\frac{\delta_{m+1}}{p}}\d x \esssup_Q w_m^{-\frac{\delta_{m+1}}{p_m}} 
\\
\le
[w]_{A_{\vec{p},\vec{r}}}^{\delta_{m+1}}\,\Big(\prod_{i=1}^{m-1}\essinf_{Q} w_i^{\frac1{p_i}}\Big)^{\delta_{m+1}}
\le 
[w]_{A_{\vec{p},\vec{r}}}^{\varrho}\, \essinf_{Q}\widehat{w}.
\end{multline*}

\medskip

\noindent\textbf{Case 2:} $\sum_{i=1}^{m-1}\frac1{\delta_i}=0$ (i.e., $p_j=r_j$ for $1\le j\le m-1$), and $\delta_m^{-1}\neq 0$. 

In this case $\frac1{\varrho}=\frac1{\delta_{m+1}}+ \frac1{\delta_m}=\frac{1-r}{r}$ and also we need to check that $\widehat{w}\in A_1$. To show this we use Hölder's inequality with 
$\frac{\delta_{m+1}}{\varrho} =1+\frac{\delta_{m+1}}{\delta_m}>1$ to obtain 
\begin{multline*}
\dashint_Q \widehat{w} \d x
=
\dashint_Q w^{\frac{\varrho}{p}}w_m^{-\frac{\varrho }{p_m}} \d x
\le
\Big(\dashint_Q w^{\frac{\delta_{m+1}}{p}}\d x \Big)^{\frac{\varrho}{\delta_{m+1}}}
\Big(\dashint_Q w_m^{-\frac{\delta_m}{p_m}}\d x \Big)^{\frac{\varrho}{\delta_m}}
\\
\le
[w]_{A_{\vec{p},\vec{r}}}^{\varrho}\,\Big(\prod_{i=1}^{m-1}\essinf_{Q} w_i^{\frac1{p_i}}\Big)^{\varrho}
\le 
[w]_{A_{\vec{p},\vec{r}}}^{\varrho}\, \essinf_{Q}\widehat{w},
\end{multline*}
which proves the desired membership and bound.

\medskip

\noindent\textbf{Case 3:} $\sum_{i=1}^{m-1}\frac1{\delta_i}>0$. 

In this case 
\[
\frac{1-r}{r}
=
\sum_{i=1}^{m+1}\frac1{\delta_i}
=
\sum_{i=1}^{m-1}\frac1{\delta_i}+
\frac1{\varrho}
>
\frac1{\varrho}.
\]
Let $\mathcal{I}=\{i: 1\le i\le m-1, \delta_i^{-1}\neq 0\}\neq\emptyset$ and $\mathcal{I}'=\{1,\dots,m-1\}\setminus\mathcal{I}$.
Set
\[
\frac1{\eta_i}
:=
\frac1{\delta_i}
\Big(\sum_{j=1}^{m-1} \frac1{\delta_j}\Big)^{-1}
=
\frac1{\delta_i}
\Big(
\sum_{j=1}^{m+1} \frac1{\delta_j}-\frac1{\delta_m}-\frac1{\delta_{m+1}}\Big)^{-1}
=
\frac1{\delta_i}
\Big(\frac{1-r}{r}-\frac1{\varrho}\Big)^{-1}
\]
for every $i\in\mathcal{I}$, and note that 
$\sum_{i\in\mathcal{I}}\frac1{\eta_i}=1$. Then Hölder's inequality leads to
\begin{multline}\label{q35t5}
\Big(\dashint_Q \widehat{w}^{1-(\frac{1-r}{r}\varrho)'}\d x\Big)^{\frac{1-r}{r}\varrho-1}
=
\Big(\dashint_Q \prod_{i\in\mathcal{I}} w_i^{-\frac{\delta_i}{p_i}\frac1{\eta_i}} \prod_{i\in\mathcal{I}'} w_i^{-\frac{1}{p_i}\varrho((\frac{1-r}{r}\varrho)'-1)} \d x
\Big)^{\frac{1-r}{r}\varrho-1}
\\
\le
\prod_{i\in\mathcal{I}} \Big(\dashint_Q w_i^{-\frac{\delta_i}{p_i}} \d x
\Big)^{\frac{\varrho}{\delta_i}}
\Big(\prod_{i\in\mathcal{I}'} \esssup_{Q} w_i^{-\frac{1}{p_i}} \Big)^{\varrho}.
\end{multline}
On the other hand, if $\delta_m^{-1}\neq 0$ we can use Hölder's inequality with 
$\frac{\delta_{m+1}}{\varrho} =1+\frac{\delta_{m+1}}{\delta_m}>1$ to obtain 
\begin{equation}\label{qt53ae}
\dashint_Q \widehat{w} \d x
=
\dashint_Q w^{\frac{\varrho}{p}}w_m^{-\frac{\varrho }{p_m}} \d x
\le
\Big(\dashint_Q w^{\frac{\delta_{m+1}}{p}}\d x \Big)^{\frac{\varrho}{\delta_{m+1}}}
\Big(\dashint_Q w_m^{-\frac{\delta_m}{p_m}}\d x \Big)^{\frac{\varrho}{\delta_m}}.
\end{equation}
If  $\delta_m^{-1}=0$ then $\varrho=\delta_{m+1}$ and  
\begin{equation}\label{sgtgt}
\dashint_Q \widehat{w} \d x
=
\dashint_Q w^{\frac{\varrho}{p}}w_m^{-\frac{\varrho }{p_m}} \d x
\le
\Big(\dashint_Q w^{\frac{\delta_{m+1}}{p}}\d x \Big)^{\frac{\varrho}{\delta_{m+1}}}
\big(\esssup_{Q} w_m^{-\frac{1}{p_m}}\Big)^{\varrho}.
\end{equation}
If we now combine \eqref{q35t5} with either \eqref{qt53ae} or \eqref{sgtgt} we readily see that $\widehat{w}\in A_{\frac{1-r}{r}\varrho}$ with 
$[\widehat{w}]_{A_{\frac{1-r}{r}\varrho}}\le [\vec w]_{A_{\vec p, \vec r}}^{\varrho}$. This completes the proof of $(i.2)$.

\medskip

To see $(i.3)$ we proceed as before considering three cases.

\medskip

\noindent\textbf{Case 1:} $\sum_{i=1}^{m}\frac1{\delta_i}=0$, that is,  $p_j=r_j$ for $1\le j\le m$.

In this case  we first observe that 
\begin{equation}\label{holder-trivial-case1}
1
=
\dashint_Q \widehat{w}\widehat{w}^{-1}\d x
\le
\Big(\dashint_Q \widehat{w}\d x\Big) \esssup_{Q}\widehat{w}^{-1}
\le
\Big(\dashint_Q \widehat{w}\d x\Big) 
\prod_{i=1}^{m-1} \esssup_{Q} w_i^{-\frac{\varrho}{p_i}}.
\end{equation}
This and \eqref{qrfera} imply
\begin{multline*}
\dashint_Q W^{\frac{\delta_{m+1}}{r_m}} \d \widehat w
=
\Big(\dashint_Q \widehat{w}\d x\Big)^{-1}\Big(\dashint_Q w^{\frac{\delta_{m+1}}{p}}
\d x\Big)
\\
\le
[\vec w]_{A_{\vec p, \vec r}}^{\delta_{m+1}}
\Big(\prod_{i=1}^{m-1} \esssup_{Q} w_i^{-\frac{\delta_{m+1}}{p_i}}\Big)
\Big(\prod_{i=1}^m \esssup_{Q} w_i^{-\frac{1}{p_i}}\Big)^{-\delta_{m+1}}
\\
=
[\vec w]_{A_{\vec p, \vec r}}^{\delta_{m+1}}
\big(\esssup_{Q} w_m^{-\frac{1}{p_m}}\big)^{-\delta_{m+1}}
=
[\vec w]_{A_{\vec p, \vec r}}^{\delta_{m+1}}
\big(\esssup_{Q} W^{-\frac{\delta_{m+1}}{r_m}}  \big)^{-1},
\end{multline*}
where we have used that since in this case $\varrho=\delta_{m+1}$ and that $W=w_m^{\frac{r_m}{p_m}}=w_m$ since $p_m=r_m$. This shows that 
 $W\in A_{1,\frac{\delta_{m+1}}{r_m}}(\widehat{w})$  with  $[W]_{A_{1,\frac{\delta_{m+1}}{r_m}} (\widehat{w})} \le [\vec w]_{A_{\vec p, \vec r}}^{\delta_{m+1}}$. Notice that we have implicitly used that since $0<\widehat{w}<\infty$ a.e.  then the Lebesgue measure and  $\widehat{w}$ have the same null measure sets, hence the essential suprema and infima with respect to the Lebesgue measure and $\widehat{w}$ agree. 

\medskip

\noindent\textbf{Case 2:} $\sum_{i=1}^{m-1}\frac1{\delta_i}=0$ (i.e., $p_j=r_j$ for $1\le j\le m-1$), and $\delta_m^{-1}\neq 0$. 

In this case we use \eqref{aaa:1} and \eqref{holder-trivial-case1} to obtain the desired estimate:
\begin{align*}
&\Big(
\dashint_Q W^{\frac{\delta_{m+1}}{r_m}}   \d \widehat w
\Big)
\Big(
\dashint_Q  W^{-(\frac{p_m}{r_m})'}\d \widehat w
\Big)^{\frac{\frac{\delta_{m+1}}{r_m}}{(\frac{p_m}{r_m})'}}
\\
&
\qquad=
\Big(\dashint_Q \widehat{w} \d x\Big)^{-1-\frac{\delta_{m+1}}{\delta_m}}
\Big(\dashint_Q w^{\frac{\delta_{m+1}}{p}} \d x\Big)
\Big(
\dashint_Q  w_m^{-\frac{\delta_m}{p_m}} \d x
\Big)^{\frac{\delta_{m+1}}{\delta_m}}
\\
&
\qquad\le
\Big(\dashint_Q w^{\frac{\delta_{m+1}}{p}} \d x\Big)
\Big(
\dashint_Q  w_m^{-\frac{\delta_m}{p_m}} \d x
\Big)^{\frac{\delta_{m+1}}{\delta_m}}
\Big(\prod_{i=1}^{m-1} \esssup_{Q} w_i^{-\frac{1}{p_i}}\Big)^{\delta_{m+1}}
\\
&
\qquad\le
[\vec w]_{A_{\vec p, \vec r}}^{\delta_{m+1}}.
\end{align*}

\medskip

\noindent\textbf{Case 3:} $\sum_{i=1}^{m-1}\frac1{\delta_i}>0$.

We saw in the proof of $(i.2)$ above that $\frac{1-r}{r}\varrho>1$ hence Hölder's inequality with that exponent gives
\begin{align*}
1=
\Big(
\dashint_{Q} \widehat{w}^{\frac{r\varrho}{1-r}} \widehat{w}^{-\frac{r\varrho}{1-r}}dx\Big)^{\frac{1-r}{r}\varrho}
	\le
	\Big(\dashint_{Q} \widehat{w}dx\Big)
\Big(\dashint_Q \widehat{w}^{1-(\frac{1-r}{r}\varrho)'}\d x\Big)^{\frac{1-r}{r}\varrho-1}.
\end{align*}
This, \eqref{aaa:2} and \eqref{q35t5} yield if we further assume that  $\delta_m^{-1}\neq 0$ (that is $r_m<p_m$):
\begin{align*}
&
\bigg[\Big(
\dashint_Q W^{\frac{\delta_{m+1}}{r_m}}   \d \widehat w
\Big)
\Big(
\dashint_Q  W^{-(\frac{p_m}{r_m})'}\d \widehat w
\Big)^{\frac{\frac{\delta_{m+1}}{r_m}}{(\frac{p_m}{r_m})'}}\bigg]^{\frac1{\delta_{m+1}}}
\\
&\le
\Big(\dashint_Q \widehat{w}^{1-\big(\frac{1-r}{r}\varrho\big)'}\d x\Big)^{\frac1{\varrho}(\frac{1-r}{r}\varrho-1)}
\Big( \dashint_Q w^{\frac{\delta_{m+1}}{p}}  \d x \Big)^{\frac 1{\delta_{m+1}}}
\Big( \dashint_Q   w_m^{-\frac{\delta_m}{p_m}}\d x \Big)^{\frac1{\delta_m}}
\\
&\le
\Big( \dashint_Q w^{\frac{\delta_{m+1}}{p}}  \d x \Big)^{\frac 1{\delta_{m+1}}}
\Big( \dashint_Q   w_m^{-\frac{\delta_m}{p_m}}\d x \Big)^{\frac1{\delta_m}}
\Big(\prod_{i\in\mathcal{I}}\! \Big(\dashint_Q w_i^{-\frac{\delta_i}{p_i}} \d x
\Big)^{\frac{1}{\delta_i}}\Big)
\Big(\!\prod_{i\in\mathcal{I}'}\! \esssup_{Q} w_i^{-\frac{1}{p_i}} \Big)\!
\\
&\le 
[\vec w]_{A_{\vec p, \vec r}}.
\end{align*}
Taking the sup over all cubes we conclude as desired that $W\in A_{\frac{p_m}{r_m},\frac{\delta_{m+1}}{r_m}}(\widehat{w})$ 
with $[W]_{A_{\frac{p_m}{r_m},\frac{\delta_{m+1}}{r_m}} (\widehat{w})}
\le
[\vec w]_{A_{\vec p, \vec r}}^{\delta_{m+1}}$. On the other hand, if $\delta_m^{-1}=0$, i.e., $r_m=p_m$, we can invoke \eqref{aaa:3} and \eqref{q35t5}
\begin{multline*}
\Big(
\dashint_Q W^{\frac{\delta_{m+1}}{r_m}} \d \widehat w
\Big)
\le
\bigg[
\Big(\dashint_Q \widehat{w}^{1-\big(\frac{1-r}{r}\varrho\big)'}\d x\Big)^{\frac1{\varrho}(\frac{1-r}{r}\varrho-1)}
\Big( \dashint_Q w^{\frac{\delta_{m+1}}{p}}  \d x \Big)^{\frac 1{\delta_{m+1}}}
\bigg]^{\delta_{m+1}}
\\
\le
\bigg[
\Big( \dashint_Q w^{\frac{\delta_{m+1}}{p}}  \d x \Big)^{\frac 1{\delta_{m+1}}}
\Big(\prod_{i\in\mathcal{I}} \Big(\dashint_Q w_i^{-\frac{\delta_i}{p_i}} \d x
\Big)^{\frac{1}{\delta_i}}\Big)
\Big(\prod_{i\in\mathcal{I}'} \esssup_{Q} w_i^{-\frac{1}{p_i}} \Big)
\bigg]^{\delta_{m+1}}
\\
\le 
[\vec w]_{A_{\vec p, \vec r}}^{\delta_{m+1}}\essinf_{Q} w_m^{\frac{\delta_{m+1}}{p_m}}
=
[\vec w]_{A_{\vec p, \vec r}}^{\delta_{m+1}}\essinf_{Q}  W^{\frac{\delta_{m+1}}{r_m}}, 
\end{multline*}
since in this case $\varrho=\delta_{m+1}$ and $W=w_m^{\frac{r_m}{p_m}}=w_m$. This completes the proof of $(i.3)$ and hence that of $(i)$.

\medskip

We now turn our attention $(ii)$. Fix $w_i^{\frac {\theta_i}{p_i}}\in A_{\frac{1-r}{r}\theta_i}$,  $1\le i\le m-1$, so that $\widehat{w}\in A_{\frac{1-r}{r}\varrho}$ (see \eqref{formula2-1}); and $W\in A_{\frac{p_m}{r_m},\frac{\delta_{m+1}}{r_m}} (\widehat{w})$. Let $w_m$ be as in \eqref{formula-2}. Our goal  is to see that 
$\vec{w}\in A_{\vec p, \vec r}$ and, much as before, we split the proof in three cases:

\medskip

\noindent\textbf{Case 1:} $\sum_{i=1}^{m}\frac1{\delta_i}=0$, that is,  $p_j=r_j$ for $1\le j\le m$. 

Note that in this case $\theta_i=\frac r{1-r}$ for every $1\le i\le m-1$. This and Hölder's inequality yield
\begin{multline}\label{6awf}
\essinf_{Q} \Big(\prod_{i=1}^{m-1} w_i^{\frac1{p_i}}\Big)
\le
\Big(\dashint_{Q} \prod_{i=1}^{m-1} w_i^{\frac{\theta_i}{p_i}\frac1{m-1}}dx\Big)^{\frac{(m-1)(1-r)}{r}}
\\
\le
\prod_{i=1}^{m-1} \Big(\dashint_{Q} w_i^{\frac{\theta_i}{p_i}}dx\Big)^{\frac1{\theta_i}}
\le
\prod_{i=1}^{m-1} \Big[w_i^{\frac {\theta_i }{p_i}}\Big]_{ A_1}^{\frac1{\theta_i}}
\essinf_{Q} w_i^{\frac1{p_i}},
\end{multline}
where in the last estimate we have used that in the present scenario $\frac{1-r}{r}\theta_i=1$. This and \eqref{qrfera} give
\begin{align*}
\Big( \dashint_Q w^{\frac{\delta_{m+1}}{p}}  \d x \Big)^{\frac 1{\delta_{m+1}}}
&=
\Big(\dashint_Q W^{\frac{\delta_{m+1}}{r_m}} \d \widehat w\Big)^{\frac 1{\delta_{m+1}}}
\Big(\dashint_Q \widehat{w}\d x\Big)^{\frac 1{\delta_{m+1}}} 
\\
&\le
[W]_{A_{1,\frac{\delta_{m+1}}{r_m}} (\widehat{w})}^{\frac1{\delta_{m+1}}}
[\widehat{w}]_{A_1}^{^{\frac1{\delta_{m+1}}}}
\essinf_Q W^{\frac{1}{r_m}} 
\essinf_Q \widehat{w}^{\frac{1}{\delta_{m+1}}} 
\\
&=
[W]_{A_{1,\frac{\delta_{m+1}}{r_m}} (\widehat{w})}^{\frac1{\delta_{m+1}}}
[\widehat{w}]_{A_1}^{^{\frac1{\delta_{m+1}}}}
\essinf_Q w_m^{\frac{1}{p_m}} 
\essinf_{Q} \Big(\prod_{i=1}^{m-1} w_i^{\frac1{p_i}}\Big)
\\
&\le
[W]_{A_{1,\frac{\delta_{m+1}}{r_m}} (\widehat{w})}^{\frac1{\delta_{m+1}}}
[\widehat{w}]_{A_1}^{^{\frac1{\delta_{m+1}}}}
\Big(
\prod_{i=1}^{m-1} \Big[w_i^{\frac {\theta_i }{p_i}}\Big]_{ A_1}^{\frac1{\theta_i}}
\Big)
\Big(\prod_{i=1}^{m} \essinf_{Q} w_i^{\frac1{p_i}}\Big),
\end{align*}
where we have used that $p_m=r_m$, $w_m=W^{\frac{p_m}{r_m}}=W$ and that $\varrho=\delta_{m+1}$. This readily leads to the desired estimate.

\medskip

\noindent\textbf{Case 2:} $\sum_{i=1}^{m-1}\frac1{\delta_i}=0$ (i.e., $p_j=r_j$ for $1\le j\le m-1$), and $\delta_m^{-1}\neq 0$. 

Using \eqref{aaa:1} and \eqref{6awf} we see that
\begin{align*}
&\Big(\dashint_Q w^{\frac{\delta_{m+1}}{p}} \d x\Big)^{\frac1{\delta_{m+1}}}
\Big(
\dashint_Q  w_m^{-\frac{\delta_m}{p_m}} \d x
\Big)^{\frac1{\delta_m}}
\\
&\qquad=
\Big(
\dashint_Q W^{\frac{\delta_{m+1}}{r_m}}   \d \widehat w
\Big)^{\frac1{\delta_{m+1}}}
\Big(
\dashint_Q  W^{-(\frac{p_m}{r_m})'}\d \widehat w
\Big)^{\frac{1}{r_m(\frac{p_m}{r_m})'}}
\Big(\dashint_Q \widehat{w} \d x\Big)^{\frac1{\varrho}}
\\
&\qquad\le
[W]_{A_{\frac{p_m}{r_m},\frac{\delta_{m+1}}{r_m}} (\widehat{w})}^{\frac1{\delta_{m+1}}}
[\widehat{w}]_{A_1}^{\frac1 {\varrho}}\essinf_{Q} \Big(\prod_{i=1}^{m-1} w_i^{\frac1{p_i}}\Big)
\\
&\qquad\le
[W]_{A_{\frac{p_m}{r_m},\frac{\delta_{m+1}}{r_m}} (\widehat{w})}^{\frac1{\delta_{m+1}}}
[\widehat{w}]_{A_1}^{^{\frac1{\delta_{m+1}}}}
\Big(
\prod_{i=1}^{m-1} \Big[w_i^{\frac {\theta_i }{p_i}}\Big]_{ A_1}^{\frac1{\theta_i}}
\Big)
\Big(\prod_{i=1}^{m-1} \essinf_{Q} w_i^{\frac1{p_i}}\Big),
\end{align*}
which readily gives that $\vec{w}\in A_{\vec{p},\vec{r}}$ with the desired bound. 

\medskip

\noindent\textbf{Case 3:} $\sum_{i=1}^{m-1}\frac1{\delta_i}>0$. 

Let us set $\eta_i=\frac{1-r}{r}(m-1)\theta_i$ for $1\le i\le m-1$ and $\eta_m=(m-1)(\frac{1-r}{r}\varrho)'$ and note that
\begin{align*}
\sum_{i=1}^m\frac1{\eta_i}
&=
\frac{r}{(m-1)(r-1)}\sum_{i=1}^{m-1}\frac1{\theta_i}+ \frac1{(m-1)(\frac{1-r}{r}\varrho)'}
\\
&=
1-\frac{r}{(m-1)(r-1)}\sum_{i=1}^{m-1}\frac1{\delta_i} 
+
\frac1{m-1}-\frac1{m-1}\frac{r}{(1-r)\varrho}
\\
&=
1-\frac{r}{(m-1)(r-1)}\Big(\frac{1-r}{r}-\frac1{\varrho}\Big)
+
\frac1{m-1}-\frac1{m-1}\frac{r}{(1-r)\varrho}
\\
&
=
1.
\end{align*}
Thus Hölder's inequality with the exponents $\eta_i$, $1\le i\le m$, yields
\begin{multline*}
1=
\Big(\dashint_Q \widehat{w}^{-\frac{r}{(m-1)(1-r)\varrho}} \widehat{w}^{\frac{r}{(m-1)(1-r)\varrho}}dx\Big)^{\frac{1-r}{r}(m-1)}
\\
=
\Big(\dashint_Q \widehat{w}^{(1-(\frac{1-r}{r}\varrho)')\frac1{\eta_m}} \prod_{i=1}^{m-1} w_i^{\frac{\theta_i}{p_i}\frac1{\eta_i}}dx\Big)^{\frac{1-r}{r}(m-1)}
\\
\le
\Big(\dashint_Q \widehat{w}^{1-(\frac{1-r}{r}\varrho)'}dx\Big)^{\frac1{\varrho}(\frac{1-r}{r}\varrho-1)}
\prod_{i=1}^{m-1} \Big(\dashint_Q w_i^{\frac{\theta_i}{p_i}}dx\Big)^{\frac1{\theta_i}}.
\end{multline*}
On the other hand, if we let $\mathcal{I}=\{j: 1\le j\le m-1, \delta_j^{-1}\neq 0\}\neq\emptyset$ and $\mathcal{I}'=\{1,\dots,m-1\}\setminus\mathcal{I}$ we observe that $\frac{\theta_i}{p_i}((\frac{1-r}{r}\theta_i)'-1)=\frac{\delta_i}{p_i}$ and hence the previous estimate yields
\begin{multline*}
\Big(\prod_{i\in\mathcal{I}}  \Big(\dashint_Q w_i^{-\frac{\delta_i}{p_i}} \d x
\Big)^{\frac{1}{\delta_i}}\Big)
\Big(\prod_{i\in\mathcal{I}'}  \esssup_{Q} w_i^{-\frac1{p_i}}\Big)
\le
\prod_{i=1}^{m-1}  \Big[w_i^{\frac {\theta_i}{ p_i}}\Big]_{A_{\frac{1-r}{r}\theta_i}}^{\frac1{\theta_i}}  \Big(\dashint_Q w_i^{\frac{\theta_i}{p_i}} dx\Big)^{-\frac1{\theta_i}}
\\
\le
\Big(\prod_{i=1}^{m-1}  \Big[w_i^{\frac {\theta_i}{ p_i}}\Big]_{A_{\frac{1-r}{r}\theta_i}}^{\frac1{\theta_i}} \Big) 
\Big(\dashint_Q \widehat{w}^{1-(\frac{1-r}{r}\varrho)'}dx\Big)^{\frac1{\varrho}(\frac{1-r}{r}\varrho-1)}.
\end{multline*}
This and \eqref{aaa:2} gives when   $\frac1{\delta_m}\neq 0$  
\begin{align*}
&\Big( \dashint_Q w^{\frac{\delta_{m+1}}{p}}  \d x \Big)^{\frac 1{\delta_{m+1}}}
\Big( \dashint_Q   w_m^{-\frac{\delta_m}{p_m}}\d x \Big)^{\frac1{\delta_m}}
\Big(\prod_{i\in\mathcal{I}}  \Big(\dashint_Q w_i^{-\frac{\delta_i}{p_i}} \d x
\Big)^{\frac{1}{\delta_i}}\Big)
\Big(\prod_{i\in\mathcal{I}'}  \esssup_{Q} w_i^{-\frac1{p_i}}\Big)
\\
&\qquad\le
\Big(\prod_{i=1}^{m-1}  \Big[w_i^{\frac {\theta_i}{ p_i}}\Big]_{A_{\frac{1-r}{r}\theta_i}}^{\frac1{\theta_i}} \Big) 
\bigg[\Big(
\dashint_Q W^{\frac{\delta_{m+1}}{r_m}}   \d \widehat w
\Big)
\Big(
\dashint_Q  W^{-(\frac{p_m}{r_m})'}\d \widehat w
\Big)^{\frac{\frac{\delta_{m+1}}{r_m}}{(\frac{p_m}{r_m})'}}\bigg]^{\frac1{\delta_{m+1}}}
\\
&\hskip5cm
\times\Big(
\dashint_Q \widehat{w}\d x
\Big)^{\frac1{\varrho}}
\Big(\dashint_Q \widehat{w}^{1-(\frac{1-r}{r}\varrho)'}dx\Big)^{\frac1{\varrho}(\frac{1-r}{r}\varrho-1)}.
\\
&\qquad\le
[W]_{A_{\frac{p_m}{r_m},\frac{\delta_{m+1}}{r_m}} (\widehat{w})}^{\frac1{\delta_{m+1}}}
[\widehat{w}]_{A_{\frac{1-r}{r}\varrho}}^{\frac1 {\varrho}}
\prod_{i=1}^{m-1}\Big[w_i^{\frac {\theta_i}{ p_i}}\Big]_{A_{\frac{1-r}{r}\theta_i}}^{\frac1{\theta_i}}.
\end{align*}
On the other hand, if  $\frac1{\delta_m}= 0$  we invoke \eqref{aaa:3}:
\begin{align*}
&\Big( \dashint_Q w^{\frac{\delta_{m+1}}{p}}  \d x \Big)^{\frac 1{\delta_{m+1}}}
\big( \esssup_{Q} w_m^{-\frac1{p_m}}\big)
\Big(\prod_{i\in\mathcal{I}}  \Big(\dashint_Q w_i^{-\frac{\delta_i}{p_i}} \d x
\Big)^{\frac{1}{\delta_i}}\Big)
\Big(\prod_{i\in\mathcal{I}'}  \esssup_{Q} w_i^{-\frac1{p_i}}\Big)
\\
&\qquad\le
\Big(\prod_{i=1}^{m-1}  \Big[w_i^{\frac {\theta_i}{ p_i}}\Big]_{A_{\frac{1-r}{r}\theta_i}}^{\frac1{\theta_i}} \Big) 
\bigg[\Big(
\dashint_Q W^{\frac{\delta_{m+1}}{r_m}}   \d \widehat w\Big) \big( \esssup_{Q} W^{-\frac{\delta_{m+1}}{r_m}}  \big)
\bigg]^{\frac1{\delta_{m+1}}}
\\
&\hskip5cm
\times\Big(
\dashint_Q \widehat{w}\d x
\Big)^{\frac1{\varrho}}
\Big(\dashint_Q \widehat{w}^{1-(\frac{1-r}{r}\varrho)'}dx\Big)^{\frac1{\varrho}(\frac{1-r}{r}\varrho-1)}
\\
&\qquad\le
[W]_{A_{1,\frac{\delta_{m+1}}{r_m}} (\widehat{w})}^{\frac1{\delta_{m+1}}}
[\widehat{w}]_{A_{\frac{1-r}{r}\varrho}}^{\frac1 {\varrho}}
\prod_{i=1}^{m-1}\Big[w_i^{\frac {\theta_i}{ p_i}}\Big]_{A_{\frac{1-r}{r}\theta_i}}^{\frac1{\theta_i}},
\end{align*}
since in this case $p_m=r_m$, $\varrho=\delta_{m+1}$, and $W=w_m^{\frac{r_m}{p_m}}=w_m$. This completes the proof of $(ii)$.

\medskip

To finish we observe that \eqref{LHS-rew} and \eqref{RHS-rew} follow at once from the definition of $\varrho$ and either \eqref{formula-1} for $(i)$ or \eqref{formula-2} for $(ii)$. This completes the proof.
\end{proof}

\section{Proof of Theorem \ref{theor:extrapol-general}}\label{section:proof-main}

The proof of Theorem \ref{theor:extrapol-general} (and Remark \ref{remark:end-point})  is split in three main steps. First, we prove a restricted version on which  all the exponents remain fixed but one. For simplicity in the presentation we will fix $p_i$, $1\le i\le m-1$, and vary $p_m$. On the other hand, since we can rearrange the $f_i$'s this clearly extends to any other choice. Second, we iterate the first step to eventually pass from $\vec{p}$ to a generic $\vec{q}$. Last, we see how the easily derive the vector-valued inequalities. 

\subsection{Step 1: Extrapolation on one component}

We first prove a particular version on which we only change one component in $\vec{p}$ (say the last one). Fix  then $\vec q=(q_1,\dots, q_{m-1}, q_m)$, with  $\vec{r}\preceq\vec{q}$ and $r_m<q_m$ so that $q_i=p_i$, $1\le i\le m-1$. Letting $\vec v \in A_{\vec q,\vec{r}}$, we set $\frac1q:=\sum_{i=1}^{m}\frac1{q_i}$  and $v:=\prod_{i=1}^mv_i^{\frac{q}{q_i}}$.  Define 

\[
\frac1r:=\sum_{i=1}^{m+1} \frac1{r_i};
\qquad\quad
\frac1{p_{m+1}}:=1-\frac1p;
\qquad\quad
\frac1{q_{m+1}}:=1-\frac1q,
\]
and, for i=1,\dots, m+1, 
\[
\frac1{\delta_i}=\frac1{r_i}-\frac1{p_i},
\qquad \quad
\frac1{\widetilde{\delta}_i}=\frac1{r_i}-\frac1{q_i}.
\]
Observe that $\delta_i=\widetilde{\delta}_i$ for $1\le i\le m-1$. This means that in view of \eqref{def:varrho}
\[
\frac1{\varrho}
:=
\frac 1{r_m}-\frac 1{r_{m+1}'}+\sum_{i=1}^{m-1}\frac1{p_i}
=
\frac 1{r_m}-\frac 1{r_{m+1}'}+\sum_{i=1}^{m-1}\frac1{q_i}
=
\frac1{\delta_m}+\frac1{\delta_{m+1}}
=
\frac1{\widetilde{\delta}_m}+\frac1{\widetilde{\delta}_{m+1}}
\]
For every $1\le i \le m-1$ we set $w_i:=v_i$. We then apply Lemma \ref{lemma:main}$(i)$ to $\vec v \in A_{\vec q,\vec{r}}$ and  $(i.1)$ yields for every $1\le i\le m-1$,
\[
w_i^{\frac{\theta_i }{p_i}}
=
w_i^{\frac{\theta_i }{q_i}}\in A_{\frac{1-r}{r}\theta_i }, 
\qquad
\mbox{where}\quad
\frac1{\theta_i}
:= 
\frac{1-r}r-\frac1{\widetilde{\delta}_i};
\]
while $(i.2)$ gives
\[
\widehat{w}:=\Big(\prod_{i=1}^{m-1}w_i^{\frac1{p_i}}\Big)^{\varrho}\in A_{\frac{1-r}{r}\varrho};
\]
and finally $(i.3)$ implies that
\begin{equation}\label{aferer}
V
:=
v^{\frac{r_m}{q}} \widehat{w}^{- \frac{r_m}{\widetilde{\delta}_{m+1}}}
\in
A_{\frac{q_m}{r_m},\frac{\widetilde{\delta}_{m+1}}{r_m}}(\widehat{w}).
\end{equation} 
Notice that in particular $\widehat{w}\in A_\infty$, hence it is a doubling measure which is fixed in the rest of the argument.

Let $W\in A_{\frac{p_m}{r_m},\frac{\delta_{m+1}}{r_m}}(\widehat{w})$ be an arbitrary weight and, in concert with \eqref{formula-2}, set $w_m
:=
W^{\frac{p_m}{r_m}} \widehat{w}^{\frac{p_m}{\delta_m}}.
$
Since $w_i^{\frac{\theta_i }{p_i}}\in A_{\frac{1-r}{r}\theta_i }$ for $1\le i\le m-1$ and $\widehat{w}\in A_{\frac{1-r}{r}\varrho}$, we can apply Lemma \ref{lemma:main}$(ii)$ with $\vec{p}$ and $\vec{r}$ to see that $\vec{w}=(w_1,\dots, w_m)\in A_{\vec p,\vec{r}}$ (notice that $\varrho$ is fixed and depends on  $p_i=q_i$, $1\le i \le m-1$, $r_m$, $r_{m+1}$). Thus, by hypothesis it follows that \eqref{extrapol:H} holds. Invoking Lemma \ref{lemma:main}$(iii)$ we then  see that for every $(f,f_1,\dots,f_m)\in\mathcal{F}$
\begin{multline}\label{faefe}
\Big\|\Big(f\widehat{w}^{-\frac1{r_{m+1}'}}\Big)^{r_m}\Big\|_{L^\frac{p}{r_m}(W^{\frac{p}{r_m}}\d\widehat{w})}^{\frac1{r_m}}
=
\|f\|_{L^p(w)}
\lesssim
\prod_{i=1}^m\|f_i\|_{L^{p_i}(w_i)} 
\\
=
\Big(\prod_{i=1}^{m-1}\|f_i\|_{L^{p_i}(w_i)}\Big)
\Big\|\Big(f_m\widehat{w}^{-\frac1{r_m}}\Big)^{r_m}\Big\|_{L^\frac{p_m}{r_m}(W^{\frac{p_m}{r_m}}\d\widehat{w})}^{\frac1{r_m}}.
\end{multline}
Let us introduce
\[
\mathcal{G}
:=
\bigg\{
\bigg(\Big(f\widehat{w}^{-\frac1{r_{m+1}'}}\Big)^{r_m}, \Big(\Big(\prod_{i=1}^{m-1}\|f_i\|_{L^{p_i}(w_i)}\Big)f_m\widehat{w}^{-\frac1{r_m}}\Big)^{r_m}\bigg):
(f,f_1,\dots,f_m)\in\mathcal{F}
\bigg\},
\]
and \eqref{faefe} can be written as
\[
\|F\|_{L^\frac{p}{r_m}(W^{\frac{p}{r_m}}\d\widehat{w})}
\lesssim
\|G \|_{L^\frac{p_m}{r_m}(W^{\frac{p_m}{r_m}}\d\widehat{w})},
\qquad
\forall\,(F,G)\in\mathcal{G},
\]
which holds for every  $W\in A_{\frac{p_m}{r_m},\frac{\delta_{m+1}}{r_m}}(\widehat{w})$. This allows us to employ Theorem \ref{theor:duo} to obtain that for every $s_m$ with $r_m<s_m<\infty$ and $0<s,\tau<\infty$ such that
\begin{equation}\label{fraefwera}
\frac 1s-\frac 1{p}=\frac1{\tau}-\frac 1{\delta_{m+1}}=\frac 1{s_m}-\frac 1{p_m},
\end{equation}
and for every $U\in A_{\frac{s_m}{r_m},\frac{\tau}{r_m}} (\widehat{w})$ the following estimate holds
\begin{equation}\label{wegsrt54}
\|F\|_{L^\frac{s}{r_m}(U^{\frac{s}{r_m}}\d\widehat{w})}
\lesssim
\|G \|_{L^\frac{s_m}{r_m}(U^{\frac{s_m}{r_m}}\d\widehat{w})},
\qquad
\forall\,(F,G)\in\mathcal{G}.
\end{equation}

Next let $s=:q$, $s_m=:q_m$ and $\tau=\widetilde{\delta}_{m+1}$. Notice that by assumption $r_m<q_m=s_m$. Since $q_i=p_i$ for $1\le i \le m-1$, it follows that
\[
\frac 1s-\frac 1{p}
=
\frac 1{q}-\frac 1{p}
=
\sum_{i=1}^m \Big(\frac1{q_i}-\frac1{p_i}\Big)
=
\frac1{q_m}-\frac1{p_m}
=
\frac 1{s_m}-\frac 1{p_m}
\]
and
\[
\frac1{\tau}-\frac1{\delta_{m+1}}
=
\frac1{\widetilde{\delta}_{m+1}}-\frac1{\delta_{m+1}}
=
\frac 1{p_{m+1}}-\frac 1{q_{m+1}}
=
\frac 1{q}-\frac 1{p}
=
\frac 1{s}-\frac 1{p},
\]
thus \eqref{fraefwera} holds.	On the other hand, note that  \eqref{aferer} gives $V\in
A_{\frac{q_m}{r_m},\frac{\widetilde{\delta}_{m+1}}{r_m}}(\widehat{w})=A_{\frac{s_m}{r_m},\frac{\tau}{r_m}} (\widehat{w})
$. All these imply that \eqref{wegsrt54} holds with $U=V$ and these choices of parameters. Consequently,  Lemma \ref{lemma:main}$(iii)$ (applied with $\vec{q}$ and $\vec{r}$) yields for every $(f,f_1,\dots,f_m)\in\mathcal{F}$
\begin{multline*}
\|f\|_{L^q(v)}
=
\Big\|\Big(f\widehat{w}^{-\frac1{r_{m+1}'}}\Big)^{r_m}\Big\|_{L^\frac{q}{r_m}(V^{\frac{q}{r_m}}\d\widehat{w})}^{\frac1{r_m}}
=
\|F\|_{L^\frac{q}{r_m}(V^{\frac{q}{r_m}}\d\widehat{w})}^{\frac1{r_m}}
\lesssim
\|G \|_{L^\frac{q_m}{r_m}(V^{\frac{q_m}{r_m}}\d\widehat{w})}^{\frac1{r_m}}
\\
=
\big( \prod_{i=1}^{m-1}\|f_i\|_{L^{p_i}(w_i)}\big)
\Big\|\Big(f_m\widehat{w}^{-\frac1{r_m}}\Big)^{r_m}\Big\|_{L^\frac{q_m}{r_m}(W^{\frac{q_m}{r_m}}\d\widehat{w})}^{\frac1{r_m}}
=
\prod_{i=1}^{m}\|f_i\|_{L^{q_i}(w_i)},
\end{multline*}
which is desired estimate in the present case.

\subsection{Step 2: Extrapolation on all components}

To complete the proof of Theorem \ref{theor:extrapol-general} (and of Remark \ref{remark:end-point})  we need to extrapolate from the given $\vec{p}=(p_1,\dots,p_m)$ with $\vec{r}\preceq\vec{p}$ to an arbitrary $\vec{q}=(q_1,\dots, q_m)$ satisfying  $\vec{r}\prec\vec{q}$. In view of Remark \ref{remark:end-point} if $p_i=r_i$ for some $1\le i\le m$ we can allow $r_i\le q_i$. This means that $\vec{r}\preceq\vec{q}$ with $r_j<q_j$ for those $j$'s for which $r_j<p_j$.

Schematically, in the previous section we have shown that 
\begin{multline}\label{extrapol-notation:freeze:1}
\mbox{$\vec{t}=(t_1,\dots, t_{m-1}, t_m)$  with $\vec{r}\preceq\vec{t}$  extrapolates to}
 \\      
\mbox{ $\vec{s}=(t_1,\dots, t_{m-1}, s_m)$  whenever $\vec{r}\preceq\vec{s}$ and $r_m<s_m$}.
\end{multline}
By this we mean that if \eqref{extrapol:H} holds for the exponent $\vec{t}$ and for all $\vec w \in A_{\vec t,\vec{r}}$, then \eqref{extrapol:H}  holds for the
exponent $\vec{s}$ and for all $\vec w \in A_{\vec s,\vec{r}}$ with $r_m<s_m$. Notice that in \eqref{extrapol-notation:freeze:1} the first $m-1$ components in $\vec{t}$ and  $\vec{s}$ are frozen. Switching the roles of $f_i$ and $f_m$ for some fixed $1\le i\le m-1$ and using the same schematic notation we can freeze all the components but the $i$-th to obtain 
\begin{multline}\label{extrapol-notation:freeze:i}
\mbox{$\vec{t}=(t_1,\dots, t_{i-1}, t_i, t_{i+1},\dots,  t_m)$  with $\vec{r}\preceq\vec{t}$  extrapolates to}
 \\      
\mbox{ $\vec{s}=(t_1,\dots, t_{i-1}, s_i, t_{i+1}, \dots, t_m)$  whenever $\vec{r}\preceq\vec{s}$ and $r_i<s_i$}.
\end{multline}
To prove our desired estimates we shall iterate \eqref{extrapol-notation:freeze:i} and at any stage we need to check that new vector of exponents $\vec{s}$ satisfies  $\vec{r}\preceq\vec{s}$ and $r_i<s_i$. We consider two cases:

\medskip

\noindent\textbf{Case 1:} $p_i\le q_i$ for all $1\le i \le m$.

In this case, our first goal is to see that
\begin{equation}\label{extrapol:step1:case1}
\mbox{$\vec{p}=(p_1,\dots, p_m)$  extrapolates to $\vec{t}=(t_1,t_2, \dots, t_m)=(q_1,p_2,\dots,p_m)$}.
\end{equation} 
First, if $q_1=p_1$ there is nothing to see. Otherwise, if $p_1<q_1$ we have that $r_1\le p_1<q_1=t_1$ and also $r_i\le p_i=t_i$ for every $2\le i\le m$. Moreover, 
\[
\frac{1}{t}
=
\sum_{i=1}^m\frac1{t_i}
=
\frac{1}{q_1}+\sum_{i=2}^m\frac1{p_i}
\ge
\sum_{i=1}^m\frac{1}{q_i}=\frac 1q
>\frac1{r_{m+1}'},
\]
since $\vec{r}\preceq \vec{q}$. Thus $\vec{r}\preceq\vec{t}$ with $r_1<q_1$ in which case \eqref{extrapol-notation:freeze:i} applies with $i=1$ and \eqref{extrapol:step1:case1} follows. 

Next from the conclusion of \eqref{extrapol:step1:case1} we can extrapolate to  $\vec{s}=(s_1,\dots,s_m)=(q_1,q_2, p_3, \dots, p_m)$. If $q_2=p_2$ there is nothing to do, otherwise, $r_1\le p_1\le q_1=s_1$, $r_2\le p_2<q_2=s_2$ and $r_i\le p_i\le q_i=s_i$ for $i=3,\dots,m$ . Moreover,
\[
\frac{1}{s}
=
\sum_{i=1}^m \frac1{s_i}
=
\frac{1}{q_1}+ \frac{1}{q_2}+ \sum_{i=3}^m\frac1{p_i}
\ge
\sum_{i=1}^m\frac{1}{q_i}=\frac 1q
>\frac1{r_{m+1}'},
\]
since $\vec{r}\preceq \vec{q}$. Thus $\vec{r}\preceq\vec{s}$ with $r_2<s_2$ in which case \eqref{extrapol-notation:freeze:i} applies with $i=2$ and we have the desired weighted estimates with the exponent $\vec{s}=(s_1,\dots,s_m)=(q_1,q_2, p_3, \dots, p_m)$. We iterate this procedure with $i=3, \dots, m$ and in the last step we analogously pass from $(q_1,\dots, q_{m-1}, p_m)$ to $(q_1,\dots,q_m)$. The proof of the present case is then complete. 

\medskip

\noindent\textbf{Case 2:} There exists some $i$ such that $p_{i}>q_{i}$. In this case, rearranging the terms if needed we may assume that $p_i> q_i$ for $1\le i\le i_0$ and $p_i\le q_i$ for $i_0+1\le i_0\le m$ where $1\le i_0\le m$ (if $i_0=m$ we just have $p_i>q_i$ for all $1\le i\le m$).

We proceed as before and iterate \eqref{extrapol-notation:freeze:i} so that in the $i$-th step we pass from $\vec{t}=(t_1,\dots, t_m)$ to $\vec{s}=(s_1,\dots, s_m)$ where
$t_j=s_j=q_j$ for $1\le j\le i-1$, $t_i=p_i$, $s_i=q_i$, and $t_j=s_j=p_j$ for $i+1\le j\le m$. We may assume that $p_i\neq q_i$, for otherwise there is nothing to prove. Note that since $r_j\le p_j, q_j$ then clearly $r_j\le s_j$ for every $1\le j\le m$. To justify that we can invoke \eqref{extrapol-notation:freeze:i} we consider two cases. 

First, if $1\le i\le i_0$ we note that $r_i\le q_i<p_i$, hence $p_i\neq r_i$ in which case as explained above $r_i<q_i$ (see Remark \ref{remark:end-point}). Moreover, since $1\le i\le i_0$,
\[
\frac{1}{s}
=
\sum_{j=1}^m \frac1{s_j}
=
\sum_{j=1}^{i} \frac1{q_j}
+
\sum_{j=i+1}^m \frac1{p_j}
>
\sum_{j=1}^m \frac1{p_j}
=
\frac 1p
>\frac1{r_{m+1}'}.
\]
Altogether, we have seen that $\vec{r}\preceq\vec{s}$ and $r_i<q_i$. Thus we can invoke \eqref{extrapol-notation:freeze:i}.

Consider next the case $i_0+1\le i\le m$ (if $i_0=m$ this case is vacuous). In this scenario, $r_i\le p_i<q_i$ (recall that we have disregarded the trivial case $q_i=p_i$). In addition, since $i_0+1\le i\le m$,
\[
\frac{1}{s}
=
\sum_{j=1}^m \frac1{s_j}
=
\sum_{j=1}^{i} \frac1{q_j}
+
\sum_{j=i+1}^m \frac1{p_j}
\ge
\sum_{j=1}^m \frac1{q_j}
=
\frac 1q
>\frac1{r_{m+1}'}.
\]
Thus, $\vec{r}\preceq\vec{s}$ and $r_i<q_i$ which justify the use of \eqref{extrapol-notation:freeze:i}.

In both scenarios we can then perform the $i$-th step of the iteration and this completes the proof of \eqref{extrapol:C} for a generic $\vec{q}$.

\medskip

\subsection{Step 3: Vector-valued inequalities}\label{section:proof:v-v}

We now turn our attention to the vector-valued inequalities. Fix  $\vec{s}=(s_1,\dots,s_m)$ with $\vec{r}\prec\vec{s}$  where $\frac1s:=\sum_{i=1}^m\frac1{s_i}$. Define a new family $\mathcal{F}_{\vec{s}}$ consisting on the $m+1$-tuples of the form
\[
(F, F_1, \dots,F_m)=\bigg(\Big(\sum_j (f^j)^{s}\Big)^\frac1{s},\Big(\sum_j (f_1^j)^{s_1}\Big)^\frac1{s_1},\dots,\Big(\sum_j (f_m^j)^{s_m}\Big)^\frac1{s_m}\bigg)
\]
where $\{(f^j, f_1^j,\dots, f_m^j)\}_j\subset\mathcal{F}$. Without loss of generality we may assume that all of the sums in the definition of
$\mathcal{F}_{\vec{s}}$ are finite; the conclusion for infinite sums follows at once from the monotone convergence
theorem. For any $v\in A_{\vec{s},\vec{r}}$, write $v:=\prod_{i=1}^mv_i^{\frac{s}{s_i}}$  and apply \eqref{extrapol:C} with $\vec{s}$ in place of $\vec{q}$ and H\"older's inequality to obtain
\begin{multline}\label{vv-initial}
\|F\|_{L^s(v)}
=
\bigg(\sum_j \|f^j\|_{L^s(v)}^s\bigg)^\frac1s
\lesssim
\bigg(\sum_j \prod_{i=1}^m\|f_i^j\|_{L^{s_i}(v_i)}^s\bigg)^\frac1s
\\
\le
\prod_{i=1}^m\bigg(\sum_j \|f_i^j\|_{L^{s_i}(v_i)}^{s_i}\bigg)^\frac1{s_i}
=
\prod_{i=1}^m\|F_i\|_{L^{s_i}(v_i)},
\end{multline}
for every $(F,F_1,\dots,F_m)\in\mathcal{F}_{\vec{s}}$.
We can now apply the first part of Theorem \ref{theor:extrapol-general} to $\mathcal{F}_{\vec{s}}$ where we use as our initial estimate \eqref{vv-initial} in place of \eqref{extrapol:H}. Thus, \eqref{extrapol:C} holds for $\mathcal{F}_{\vec{s}}$ and this gives us immediately \eqref{extrapol:vv}. 
This completes the proof of Theorem \ref{theor:extrapol-general}. \qed

\subsection{Proof of Remark \ref{remark:iteration}}\label{section:iterarion} 

The iterated vector-valued inequality in Remark \ref{remark:iteration} follow easily by repeating the argument in the previous section. Fixed $\vec{s}$ and $\vec{t}$ and with $\mathcal{F}_{\vec{s}}$ as defined in the previous section we consider a new family  $(\mathcal{F}_{\vec{s}})_{\vec{t}}$ consisting on the $(m+1)$-tuples
\[
(\mathfrak{F}, \mathfrak{F}_1,\dots, \mathfrak{F}_m)=
\bigg(\Big(\sum_j (F^j)^{t}\Big)^\frac1{t},\Big(\sum_j (F_1^j)^{t_1}\Big)^\frac1{t_1},\dots,\Big(\sum_j (F_m^j)^{t_m}\Big)^\frac1{t_m}\bigg)
\]
where $\{(F^j, F_1^j,\dots, F_m^j)\}_j\subset\mathcal{F}_{\vec{s}}$. Without loss of generality we may assume that all of the sums in the definition of
$(\mathcal{F}_{\vec{s}})_{\vec{t}}$ are finite; the conclusion for infinite sums follows at once from the monotone convergence
theorem. Next for any $v\in A_{\vec{t},\vec{r}}$, write $v:=\prod_{i=1}^mv_i^{\frac{t}{t_i}}$  and apply \eqref{extrapol:vv} with $\vec{t}$ in place of $\vec{q}$ and H\"older's inequality to obtain
\begin{multline}\label{vv-initial:iterated}
\|\mathfrak{F}\|_{L^t(v)}
=
\bigg(\sum_j \|F^j\|_{L^t(v)}^t\bigg)^\frac1t
\lesssim
\bigg(\sum_j \prod_{i=1}^m\|F_i^j\|_{L^{t_i}(v_i)}^t\bigg)^\frac1s
\\
\le
\prod_{i=1}^m\bigg(\sum_j \|F_i^j\|_{L^{t_i}(v_i)}^{t_i}\bigg)^\frac1{t_i}
=
\prod_{i=1}^m\|\mathfrak{F}_i\|_{L^{t_i}(v_i)},
\end{multline}
for every $(\mathfrak{F},\mathfrak{F}_1,\dots,\mathfrak{F}_m)\in(\mathcal{F}_{\vec{s}})_{\vec{t}}$.
We can now apply the first part of Theorem \ref{theor:extrapol-general} to $(\mathcal{F}_{\vec{s}})_{\vec{t}}$ where we use as our initial estimate \eqref{vv-initial:iterated} in place of \eqref{extrapol:H}. Thus, \eqref{extrapol:C} holds for $(\mathcal{F}_{\vec{s}})_{\vec{t}}$ and this gives us immediately \eqref{vv-iterated}. Note that repeating this idea one can easily  obtain iterated vector-valued inequalities with arbitrary number of ``sums''. \qed

\section{Proof of Theorem \ref{thm:multilinear:comm}}\label{section:proof-comm}

We need the following auxiliary result in the spirit of \cite{BMMST}.

\begin{Proposition}\label{prop:comm-Banach}
Let $T$ be an $m$-linear operator and let $\vec{r}=(r_1,\dots,r_{m+1})$, with $1\le r_1,\dots,r_{m+1}<\infty$. Assume that there exists $\vec s=(s_1,\dots, s_m)$, with $1\le s_1,\dots, s_m<\infty$, $1<s<\infty$, and $\vec{r}\prec\vec{s}$, such that for all $\vec{w} = (w_1, \dots, w_m)\in A_{\vec{s},\vec{r}}$, we have
\begin{equation}
\label{vector-weigh:Banach}
\|T(f_1, f_2,\dots, f_m)\|_{L^s (w)} \lesssim \prod_{i=1}^m \|f_i\|_{L^{s_i}\left(w_i\right)},
\end{equation}
where $\frac1s:=\frac1{s_1}+\dots+\frac1{s_m}$ and $w:=\prod_{i=1}^m w_i^{\frac{s}{s_i}}$.

Then, 	for all weights $\vec v \in A_{\vec s,\vec{r}}$, for all $\textbf{b} = (b_1, \dots, b_m) \in  {\rm BMO}^m$, and for each multi-index $\alpha$, we have
\begin{equation}
\label{multi-commutator-II:Banach}
\|[T, \textbf{b}]_\alpha (f_1, f_2,\dots, f_m)\|_{L^s (v)}
\lesssim
\prod_{i=1}^m \|b_j\|^{\alpha_i}_{{\rm BMO}} \|f_i\|_{L^{s_i}\left(v_i\right)},
\end{equation}
where $v:=\prod_{i=1}^m v_i^{\frac{s}{s_i}}$.
\end{Proposition}

Assuming this result momentarily and adopting  the notation introduced just before Lemma \ref{lemma:main} we let $\vec{s}=(s_1,\dots,s_m)$ with 
set $s_i=\frac{r_i}{r}$, $1\le i\le m$. Since $r<1$ it follows at once that $r_i<s_i$ for every $1\le i\le m$ and 
\[
\frac1s:=
\sum_{i=1}^m \frac1{s_i}
=
r\sum_{i=1}^m \frac1{r_i}
=
r\left(\frac1r-\frac1{r_{m+1}}\right)
>
1-\frac1{r_{m+1}}
=
\frac1{r_{m+1}'}.
\]
Hence, $\vec{r}\prec\vec{s}$ and we can invoke Theorem \ref{theor:extrapol-general} to show that 
\[
\|T(f_1, f_2,\dots, f_m)\|_{L^s (w)} \lesssim \prod_{i=1}^m \|f_i\|_{L^{s_i}\left(w_i\right)}
\]
for all  $\vec{w} = (w_1, \dots, w_m)\in A_{\vec{s},\vec{r}}$, where $w:=\prod_{i=1}^mw_i^{\frac{s}{s_i}}$.

At this point we observe that by construction $s>1$, therefore Proposition \ref{prop:comm-Banach} applies and we conclude that \eqref{multi-commutator-II:Banach} holds for all  $\vec v \in A_{\vec s,\vec{r}}$. To remove the restriction $s>1$ we apply again Theorem  \ref{theor:extrapol-general} and \eqref{multi-commutator-II} and \eqref{multi-comm:VV} follows. The proof is complete modulo that of Proposition \ref{prop:comm-Banach}. \qed

\bigskip

We state some auxiliary result similar to Lemma \ref{lemma:main} (see \cite[Theorem 3.6]{LOPTT} for the case $\vec{r}=(1,\dots,1)$). Then proof is given at the end of this section. 

\begin{Lemma}\label{lemma:main:II}
	Let  $\vec p=(p_1,\dots, p_m)$ with $1< p_1,\dots,p_m<\infty$ and $\vec{r}=(r_1,\dots,r_{m+1})$
	with $1\le r_1,\dots,r_{m+1}<\infty$ be such that $\vec{r}\preceq\vec{p}$. Set, for every $1\le i\le m$,
	\begin{equation}\label{def:theta:II}
	\frac1{\theta_i}
	:= 
	\frac{1-r}r-\frac1{\delta_i} 
	=
	\sum_{j=1}^{m+1} \frac1{\delta_j}-\frac1{\delta_i}
	>0,
	\end{equation}
	where we are using the notation introduced before Lemma \ref{lemma:main}.
	Then the following hold:
	\begin{list}{$(\theenumi)$}{\usecounter{enumi}\leftmargin=.8cm
			\labelwidth=.8cm\itemsep=0.2cm\topsep=.1cm
			\renewcommand{\theenumi}{\roman{enumi}}}
		
		\item Given  $\vec{w}=(w_1,\dots,w_m)\in A_{\vec p, \vec r}$, write $w:=\prod_{i=1}^mw_i^{\frac{p}{p_i}}$. Then  $w_i^{\frac{\theta_i }{p_i}}\in A_{\frac{1-r}{r}\theta_i }$ with $\Big[w_i^{\frac {\theta_i }{p_i}}\Big]_{ A_{\frac{1-r}{r}\theta_i }} \le [\vec w]_{A_{\vec p, \vec r}}^{\theta_i }$, for every $1\le i\le m$, and  		
		$w^{\frac{\delta_{m+1}}{p}}\in A_{\frac{1-r}{r}\delta_{m+1}}$ with 
			$\Big[w^{\frac{\delta_{m+1}}{p}}\Big]_{A_{\frac{1-r}{r}\delta_{m+1}}}\le [\vec w]_{A_{\vec p, \vec r}}^{\delta_{m+1}}$.

		\item Given $w_i^{\frac {\theta_i}{p_i}}\in A_{\frac{1-r}{r}\theta_i}$,  $1\le i\le m$,  write $w:=\prod_{i=1}^mw_i^{\frac{p}{p_i}}$ and assume that $w^{\frac{\delta_{m+1}}{p}}\in A_{\frac{1-r}{r}\delta_{m+1}}$. 
		Then $\vec{w}=(w_1,\dots, w_m)\in A_{\vec p, \vec r}$ and, moreover,
		\[
		[\vec w]_{A_{\vec p, \vec r}}
		\le 
		\Big[w^{\frac{\delta_{m+1}}{p}}\Big]_{A_{\frac{1-r}{r}\delta_{m+1}}}^{\frac1 {\delta_{m+1}}}
		\prod_{i=1}^{m}\Big[w_i^{\frac {\theta_i}{ p_i}}\Big]_{A_{\frac{1-r}{r}\theta_i}}^{\frac1{\theta_i}}
		\]
		\end{list}
\end{Lemma}

\begin{proof}[Proof of Proposition \ref{prop:comm-Banach}]
The proof is a modification of \cite[Proof of Theorem 4.13]{BMMST} and we only point out the main changes. As there we introduce
\[
\|h\|_{ \BMO} := \sup_Q \|h - h_Q\|_{\exp L, Q}
=
\sup_Q 
\inf \left\{ \lambda > 0 : \dashint_Q \bigg(e^{\frac{|h(x)-h_Q|}{\lambda}}-1\bigg) dx \leq 1 \right\}.
\]
Note that by the John-Nirenberg inequality $\|h\|_{{\rm BMO}}\le \|h\|_{\BMO}\le C_n \|h\|_{{\rm BMO}}$. Without loss of generality, we assume that $b_i$, $1\le i\le m$, are real valued and normalized so that $\|b_i\|_{\BMO}=1$. By following the argument in \cite[Proof of Theorem 4.3]{BMMST} with the Cauchy integral trick, given $\vec{v}\in A_{\vec{s},\vec{r}}$ we write $v:=\prod_{i=1}^m v_i^{\frac{s}{s_i}}$ and  one can see that everything reduces to showing that for some appropriate $\gamma_1,\dots, \gamma_m>0$ (to be chosen later) and for $|z_1|=\gamma_1,\dots, |z_m|=\gamma_m$, we have that $\vec{w}\in A_{\vec{s},\vec{r}}$ where 
\[
\vec{w}:=
(w_1,\dots,w_m)
:=
(v_1 e_{b_1}, \dots, v_m e_{b_m})
:=
(v_1 e^{-\text{Re}(z_1)s_1 b_1} , \dots, v_m e^{-\text{Re}(z_m)s_m b_m}).
\]
By Lemma \ref{lemma:main:II}$(i)$, applied to $\vec{v}\in A_{\vec{s},\vec{r}}$, it follows that $v^{\frac{\delta_{m+1}}{s}}\in A_{\frac{1-r}{r}\delta_{m+1}}$ and $v_i^{\frac{\theta_i }{s_i}}\in A_{\frac{1-r}{r}\theta_i }$ or, equivalently, $v_i^{-\frac{\delta_i}{s_i}}\in A_{\frac{1-r}{r}\delta_i}$ for $1\le i\le m$.
Using \cite[Lemma 3.28]{Perez-course} for any of these weights, there exists $\eta>1$ with
\begin{equation}\label{eq:afrfr}
\eta'\sim \max
\Big\{
\Big[v^{\frac{\delta_{m+1}}{s}}\Big]_{A_{\frac{1-r}{r}\delta_{m+1}}},
\Big[v_i^{-\frac{\delta_1}{s_1}}\Big]_{ A_{\frac{1-r}{r}\delta_1 }},\dots,\Big[v_m^{-\frac {\delta_m }{s_m}}\Big]_{ A_{\frac{1-r}{r}\delta_m }}\Big\}
\le
[\vec v]_{A_{\vec s, \vec r}}^{\max\{\delta_1,\dots,\delta_{m+1}\}}
\end{equation}
so that the following reverse H\"older inequalities hold:
\begin{equation}\label{rh1}
\left( \dashint_Q v^{\frac{\delta_{m+1}}{s}\eta} dx \right)^{\frac1\eta} \leq 2 \dashint_Q v^{\frac{\delta_{m+1}}{s}} dx
\end{equation}
and, for $i=1,\dots, m$,
\begin{equation}\label{rh2}
\left( \dashint_Q v_i^{-\frac{\delta_i}{s_i}\eta}\,dx \right)^{\frac1\eta} \leq 2 \dashint_Qv_i^{-\frac{\delta_i}{s_i}}\,dx.
\end{equation}
Writing $w:=\prod_{i=1}^m w_i^{\frac{s}{s_i}}$, using the previous estimates and H\"older's inequality with $\sum_{i=1}^m \frac{s}{s_i}=1$ we get
\begin{align*}
&
\Big(\dashint_Q w^{\frac {\delta_{m+1}}{s}}\d x\Big)^{\frac 1{\delta_{m+1}}}
\prod_{i=1}^m \Big(\dashint_Q w_i^{-\frac{\delta_i}{s_i}}\d x\Big)^{\frac1{\delta_i}}
\\
&\qquad\qquad=
\left( \dashint_Q v^{\frac {\delta_{m+1}}{s}} \prod_{i=1}^m e_{b_i}^{ \frac {\delta_{m+1}}{s_i}} dx \right)^{\frac 1{\delta_{m+1}}}
\prod_{i=1}^m \Big(\dashint_Q v_i^{-\frac{\delta_i}{s_i}} e_{b_i}^{-\frac{\delta_i}{s_i}} \d x\Big)^{\frac1{\delta_i}}
\\
&\qquad\qquad\le
\left( \dashint_Q v^{\frac {\delta_{m+1}}{s}\eta} dx \right)^{\frac 1{\delta_{m+1}\eta}}
\prod_{i=1}^m \Big(\dashint_Q v_i^{-\frac{\delta_i}{s_i}\eta} \d x\Big)^{\frac1{\delta_i\eta}}
\\
&\qquad\qquad\quad\qquad
\left( \dashint_Q \prod_{i=1}^m e_{b_i}^{ \frac {\delta_{m+1}}{s_i}\eta'} dx \right)^{\frac 1{\delta_{m+1}\eta'}}
\prod_{i=1}^m \Big(\dashint_Q e_{b_i}^{-\frac{\delta_i}{s_i}\eta'} \d x\Big)^{\frac1{\delta_i\eta'}}
\\
&
\qquad\qquad\le
2^{\frac{1-r}{r}}[\vec{v}]_{A_{\vec{s},\vec{r}}}
\prod_{i=1}^m 
\left( \dashint_Q e_{b_i}^{ \frac {\delta_{m+1}}{s}\eta'} dx \right)^{\frac s{\delta_{m+1}\eta's_i}}
\Big(\dashint_Q e_{b_i}^{-\frac{\delta_i}{s_i}\eta'} \d x\Big)^{\frac1{\delta_i\eta'}}
\\
&
\qquad\qquad\le
2^{\frac{1-r}{r}}[\vec{v}]_{A_{\vec{s},\vec{r}}}
\prod_{i=1}^m \Big[e_{b_i}^{ \frac {\delta_{m+1}}{s}\eta'}\Big]_{A_{1+\frac{\delta_{m+1}s_i}{\delta_i s}}}^{\frac s{\delta_{m+1}\eta's_i}}
\\
&
\qquad\qquad\le
2^{\frac{1-r}{r}+2\sum_{i=1}^m \gamma_i}[\vec{v}]_{A_{\vec{s},\vec{r}}},
\end{align*}
where the last estimate holds provided 
\[
\gamma_i\le \frac1{\eta'}\min \Big\{\frac1{\delta_i},\frac{s}{\delta_{m+1}s_i}\Big\},
\]
and where we have used that for every $1\le q<\infty$, $\lambda\in\re$ and $h\in {\rm BMO}$ we have that 
$$
\big[e^{\lambda\,h}\big]_{A_q} \leq 4^{|\lambda|\,\|h\|_{\BMO}},
\qquad|\lambda| \leq \frac{\min\left\{1, q-1 \right\}}{\|h\|_{\BMO}},
$$
see \cite[Lemma 3.5]{BMMST}. We have eventually shown that $\vec{w}\in A_{\vec{s},\vec{r}}$. From here the argument in \cite[Proof of Theorem 4.13]{BMMST}
goes through and we can conclude the desired estimate, further details are left to the interested reader.
\end{proof}

\begin{proof}[Proof of Lemma \ref{lemma:main:II}]
We start with $(i)$. Note that in Lemma \ref{lemma:main}$(i.1)$ we have already shown that  $w_i^{\frac{\theta_i }{p_i}}\in A_{\frac{1-r}{r}\theta_i }$ with $\Big[w_i^{\frac {\theta_i }{p_i}}\Big]_{ A_{\frac{1-r}{r}\theta_i }} \le [\vec w]_{A_{\vec p, \vec r}}^{\theta_i }$, for every $1\le i\le m-1$. Notice however that the proof works in the very same way for the case $i=m$. Thus, it remains to show that $w^{\frac{\delta_{m+1}}{p}}\in A_{\frac{1-r}{r}\delta_{m+1}}$. To proceed let us introcuce 
$\mathcal I:=\{1\le i \le m: \delta_i^{-1}\neq 0\}$ and $\mathcal{I}':=\{1,\dots,m\}\setminus\mathcal{I}$. If $\mathcal I=\emptyset$, then $\frac {1-r}r= \frac 1{\delta_{m+1}}$ and trivially 
\begin{multline*}
\Big[w^{\frac{\delta_{m+1}}{p}}\Big]_{A_{\frac{1-r}{r}\delta_{m+1}}}
=
\Big[w^{\frac{\delta_{m+1}}{p}}\Big]_{A_1}
=
\sup_Q\Big(\dashint_Q w^{\frac{\delta_{m+1}}{p}}\d x\Big)
\mathop{\rm{ess\, sup}}_{x\in Q} w(x)^{-\frac{\delta_{m+1}}{p}}\\
\le \sup_Q\Big(\dashint_Q w^{\frac{\delta_{m+1}}{p}}\d x\Big)
\prod_{i=1}^m\mathop{\rm{ess\, sup}}_{x\in Q} w_i(x)^{-\frac{\delta_{m+1}}{p_i}}
\le
[\vec{w}]_{A_{\vec{p},\vec{r}}}^{\delta_{m+1}}.
\end{multline*}
Consider next the case $\mathcal I\neq\emptyset$. For $i\in \mathcal I$ let us set 
\[
\frac1{\eta_i}
:=
\frac1{\delta_i}
\Big(\sum_{j=1}^{m} \frac1{\delta_j}\Big)^{-1}
=
\frac1{\delta_i}
\Big(
\sum_{j=1}^{m+1} \frac1{\delta_j}-\frac1{\delta_{m+1}}\Big)^{-1}
=
\frac1{\delta_i}
\Big(\frac{1-r}{r}-\frac1{\delta_{m+1}}\Big)^{-1}
\]
and note that $\sum_{i\in\mathcal{I}}\frac1{\eta_i}=1$. Then Hölder's inequality easily leads to the desired estimate:
\begin{align*}
&\Big[w^{\frac{\delta_{m+1}}{p}}\Big]_{A_{\frac{1-r}{r}\delta_{m+1}}}
=
\sup_Q\Big(\dashint_Q w^{\frac{\delta_{m+1}}{p}}\d x\Big)
\Big(\dashint_Q w^{\frac{\delta_{m+1}}{p}(1-(\frac{1-r}{r}\delta_{m+1})')}\d x\Big)^{\frac{1-r}{r}\delta_{m+1}-1}
\\
&\quad\le
\sup_Q\Big(\dashint_Q w^{\frac{\delta_{m+1}}{p}}\d x\Big)
\Big(\dashint_Q \prod_{i\in \mathcal I} w_i^{-\frac{\delta_{i}}{p_i\eta_i}}\d x\Big)^{\frac{1-r}{r}\delta_{m+1}-1}\prod_{i\in \mathcal I'}\mathop{\rm{ess\, sup}}_{x\in Q} w_i(x)^{-\frac{\delta_{m+1}}{p_i}}
\\
&\quad\le
\sup_Q\Big(\dashint_Q w^{\frac{\delta_{m+1}}{p}}\d x\Big)
\Big(\prod_{i\in \mathcal I} \Big(\dashint_Q w_i^{-\frac{\delta_i}{p_i}} \d x
\Big)^{\frac{\delta_{m+1}}{\delta_i}}\Big)\prod_{i\in \mathcal I'}\mathop{\rm{ess\, sup}}_{x\in Q} w_i(x)^{-\frac{\delta_{m+1}}{p_i}}
\\
&\quad\le
[\vec{w}]_{A_{\vec{p},\vec{r}}}^{\delta_{m+1}},
\end{align*}
and this completes the proof of $(i)$.

Let us now obtain $(ii)$. Assume first that $\mathcal{I}=\emptyset$, thus for every $1\le i\le m$ we have $\delta_i^{-1}=0$ and 
$\theta_i=\frac r{1-r}=\delta_{m+1} $. Then Hölder's inequality gives
\begin{multline*}
\essinf_Q w^{\frac1p}
\le
\Big(\dashint_{Q} w^{\frac{\delta_{m+1}}{m p}}dx\Big)^{\frac1{m\delta_{m+1}}}
=
\Big(\dashint_{Q} \prod_{i=1}^{m} w_i^{\frac{\theta_i}{p_i}\frac1{m}}dx\Big)^{\frac1{m{\delta_{m+1}}}}
\\
\le
\prod_{i=1}^{m} \Big(\dashint_{Q} w_i^{\frac{\theta_i}{p_i}}dx\Big)^{\frac1{{\delta_{m+1}} }}
\le
\prod_{i=1}^{m} \Big[w_i^{\frac {\theta_i }{p_i}}\Big]_{ A_1}^{\frac1{\theta_i }}
\essinf_{Q} w_i^{\frac1{p_i}},
\end{multline*}
and thus
\begin{align*}
[\vec{w}]_{A_{\vec{p},\vec{r}}}
&= 
\sup_Q\Big(\dashint_Q w^{\frac{\delta_{m+1}}{p}}\d x\Big)^{\frac 1{\delta_{m+1}}}
\prod_{i=1}^m\mathop{\rm{ess\, sup}}_{x\in Q} w_i(x)^{-\frac{1}{p_i}}
\\
&\le
\Big[w^{\frac{\delta_{m+1}}{p}}\Big]_{A_1}^{\frac 1{\delta_{m+1}}}
\big(\essinf_Q w^{\frac1p}\big)
\prod_{i=1}^m\mathop{\rm{ess\, sup}}_{x\in Q} w_i(x)^{-\frac{1}{p_i}}
\\
&\le 
\Big[w^{\frac{\delta_{m+1}}{p}}\Big]_{A_1}^{\frac 1{\delta_{m+1}}}
\prod_{i=1}^{m} \Big[w_i^{\frac {\theta_i }{p_i}}\Big]_{ A_1}^{\frac1{\theta_i }}.
\end{align*}

Next we consider the case when $\mathcal I \neq\emptyset$. Set 
\[
\frac1{\theta_{m+1}}
:= 
\frac{1-r}r-\frac1{\delta_{m+1}} 
=
\sum_{j=1}^{m+1} \frac1{\delta_j}-\frac1{\delta_{m+1}}
>0,
\]
Since, $\sum_{i=1}^{m+1}\frac1{\theta_i}=\frac{m(1-r)}{r}$, Hölder's inequality  easily gives
\begin{align*}
1&=
\Big(\dashint_Q w^{-\frac{1}{p}\frac{r}{m(1-r)}} w^{\frac{1}{p}\frac{r}{m(1-r)}} dx\Big)^{\frac{m(1-r)}{r}}
\\
&=
\Big(\dashint_Q w^{-\frac{1}{p}\frac{r}{m(1-r)}} \prod_{i=1}^m w_i^{\frac{1}{p_i}\frac{r}{m(1-r)}} dx\Big)^{\frac{m(1-r)}{r}}
\\
&\le
\Big(\dashint_Q w^{-\frac{\theta_{m+1}}{p}}dx\Big)^{\frac1{\theta_{m+1}}} \prod_{i=1}^m \Big(\dashint_Q w_i^{\frac{\theta_i}{p_i}} dx\Big)^{\frac1{\theta_i}}
\\
&=\Big(\dashint_Q w^{\frac{\delta_{m+1}}{p}(1-(\frac{1-r}{r}\delta_{m+1})')}\d x\Big)^{\frac1{\delta_{m+1}}(\frac{1-r}{r}\delta_{m+1}-1)}\prod_{i=1}^m \Big(\dashint_Q w_i^{\frac{\theta_i}{p_i}} dx\Big)^{\frac1{\theta_i}}.
\end{align*}
This and our assumptions give the desired estimate
\begin{align*}
[\vec{w}]_{A_{\vec{p},\vec{r}}}&=
\sup_Q
\Big( \dashint_Q w^{\frac{\delta_{m+1}}{p}}  \d x \Big)^{\frac 1{\delta_{m+1}}}
\Big(\prod_{i\in \mathcal I} \Big(\dashint_Q w_i^{-\frac{\delta_i}{p_i}} \d x
\Big)^{\frac{1}{\delta_i}}\Big)\prod_{i\in \mathcal I'}\mathop{\rm{ess\, sup}}_{x\in Q} w_i(x)^{-\frac{1}{p_i}}
\\
&\le
\Big[w^{\frac{\delta_{m+1}}{p}}\Big]_{A_{\frac{1-r}{r}\delta_{m+1}}}^{\frac1{\delta_{m+1}}}\prod_{i=1}^{m}  \Big[w_i^{\frac {\theta_i}{ p_i}}\Big]_{A_{\frac{1-r}{r}\theta_i}}^{\frac1{\theta_i}}  \\
&
\qquad\times
\sup_Q
\Big(\dashint_Q w^{\frac{\delta_{m+1}}{p}(1-(\frac{1-r}{r}\delta_{m+1})')}\d x\Big)^{-\frac1{\delta_{m+1}}(\frac{1-r}{r}\delta_{m+1}-1)}\prod_{i=1}^{m}  
 \Big(\dashint_Q w_i^{\frac{\theta_i}{p_i}} dx\Big)^{-\frac1{\theta_i}}
\\
&
\le
\Big[w^{\frac{\delta_{m+1}}{p}}\Big]_{A_{\frac{1-r}{r}\delta_{m+1}}}^{\frac1{\delta_{m+1}}}
\prod_{i=1}^{m}  \Big[w_i^{\frac {\theta_i}{ p_i}}\Big]_{A_{\frac{1-r}{r}\theta_i}}^{\frac1{\theta_i}}.
\end{align*}
This completes the proof.
\end{proof}

\end{document}